\documentclass{amsart}

\usepackage{amsmath,xypic}
\usepackage{xspace}
\usepackage[psamsfonts]{amssymb}
\usepackage[latin1]{inputenc}
\usepackage{graphicx, color}

\usepackage{pgf,tikz}
\usepackage{mathrsfs}
\usetikzlibrary{arrows}

\usepackage{hyperref}

\theoremstyle{remark}
\newtheorem{example}{\bf Example}[section]

\newtheorem{exo}[example]{\bf Exercise}
\newtheorem{remark}[example]{\bf Remark}

\theoremstyle{plain}
\newtheorem{theorem}[example]{\bf Theorem}
\newtheorem{lemma}[example]{\bf Lemma}
\newtheorem{definition}[example]{\bf Definition}

\newtheorem{proposition}[example]{\bf Proposition}
\newtheorem{corollary}[example]{\bf Corollary}

\def\cal{\mathcal}

\usepackage{graphicx}        
\makeindex             

\begin{document}

\title{Introduction to Fatou components in holomorphic dynamics}
\begin{author}[X.~Buff]{Xavier Buff}
\email{xavier.buff$@$math.univ-toulouse.fr}
\address{ %
  Institut de Math\'ematiques de Toulouse\\
   UMR5219\\ Universit\'e de Toulouse, CNRS, UPS\\ F-31062 Toulouse Cedex 9\\ France }
\end{author}
\begin{author}[J.~Raissy]{Jasmin Raissy}\thanks{This work is part of the C.I.M.E. Lecture Notes ``Modern Aspects of Dynamical Systems'' to be published in the Springer Lecture Notes in Mathematics -- C.I.M.E. subseries, corresponding to the course taught by the second author in Cetraro in August 2021.}
\email{jasmin.raissy@math.u-bordeaux.fr}
\address{Univ. Bordeaux, CNRS, Bordeaux INP, IMB, UMR 5251, F-33400 Talence, France \& Institut Universitaire de France (IUF)}
\end{author}

%
%
\maketitle

%
\begin{abstract} This survey is an introduction to the classification of Fatou components in holomorphic dynamics. We start with the description of the Fatou and Julia sets for rational maps of the Riemann sphere, and finish with an updated account of the recent results on Fatou components for polynomial skew-products in complex dimension two, where we focus on the key steps in the construction giving the existence of a wandering domain for a polynomial endomorphism of $\mathbb{C}^2$.
\end{abstract}

\section{Introduction}
\label{sec:1}

\textit{Discrete dynamics} concerns maps $f:X \to X$ from a space to itself. One of the main goals is to understand sequences of the form
\[
x_0\in X, \quad x_1:= f(x_0), \quad x_2:=f(x_1)= f^{\circ 2}(x_0),\quad x_3 := f(x_2) = f^{\circ 3}(x_0),  \ldots
\]
obtained by {\em iterating $f$}.  The {\em forward orbit} of $x_0$ under $f$ is
\[{\cal O}^+(x_0) := \bigcup_{n\geq 0} \bigl\{f^{\circ n}(x_0)\bigr\}.\]

The main problem is to understand the long term behaviour of the sequence $(x_n)_{n\geq 0}$. 
The answer will depend on the assumptions we make on the space $X$ and the map $f:X\to X$. 

If $X$ is a topological space, we may wonder whether the sequence converges. If this is the case, and in addition $f:X\to X$ is a continuous map, then the limit $\ell$ is a fixed point of $f$, i.e., $f(\ell) = \ell$. 

If $X$ is compact, then thanks to the Bolzano-Weierstrass Theorem, we may extract converging subsequences. We shall denote by $\omega(x_0)$ the \emph{$\omega$-limit set of $x_0$}, that is the set of possible limit values:
\[\omega(x_0) :=\bigcap_{n\geq 0}\overline{{\cal O}^+(x_n)}.\]
If $f:X\to X$ is continuous, then for any $x\in X$, the $\omega$-limit set $\omega(x)$ is $f$-invariant. 

\begin{exo}
Prove it.
\end{exo}

In the rest of these notes, we will study the case where $X$ is a complex manifold and $f:X\to X$ is a holomorphic map. We shall first consider the case where $X=\widehat{\mathbb{C}} = \mathbb{C}\cup \{\infty\}$ is the Riemann sphere and 
$f:\widehat{\mathbb{C}}\to \widehat{\mathbb{C}}$ is a polynomial or a rational map. 

This subject has a fairly long history, with contributions by
Koenigs \cite{Koenigs}, Schr\"oder \cite{Schroder}, B\"ottcher
\cite{Bottcher} in the late 19th century, and the great memoirs of
Fatou \cite{Fatou1, Fatou2, Fatou3} and Julia
\cite{Julia} around 1920. For the history of this early period we
refer the reader to the book of Alexander \cite{Alexander}.
Followed a dormant period, with notable contributions by Cremer
\cite{Cremer1, Cremer2} (1936) and Siegel \cite{Siegel}
(1942), and a rebirth in the 1960's (Brolin, Guckenheimer,
Jakobson). Since the early 1980's, partly under the impetus of
computer graphics, the subject has grown vigorously, with major
contributions by Douady, Hubbard, Sullivan, Thurston, and more
recently Lyubich, McMullen, Milnor, Shishikura, Yoccoz \dots
Although the subject still presents many open problems, there is now
a substantial body of knowledge and
several books have appeared: Beardon \cite{Beardon},
Carleson and Gamelin \cite{CarlesonGamelin}, Milnor \cite {Milnor1}
and Steinmetz \cite{Steinmetz} and the more specialized books by
McMullen \cite{McMullen1} \cite{McMullen2}. Surveys by Blanchard
\cite{Blanchard}, Devaney \cite{Devaney}, Douady \cite{Douady1},
Keen \cite{Keen} and Lyubich \cite{Lyubich1} are also highly
recommended.

There is a small collection of polynomials, for instance
\[f(z) = az+b, \quad f(z) = z^d\quad \text{and} \quad f(z) = z^2-2,\] 
whose dynamics can be fairly easily understood. 

\begin{exo}
Assume $f(z) = az+b$ with $a\in \mathbb{C}$ and $b\in \mathbb{C}$. 
\begin{enumerate}
\item Show that when $|a|\neq 1$ or when $a=1$ and $b\neq 0$, 
the map 
\[\mathbb{C}\ni z\mapsto \omega(z)\in \widehat{\mathbb{C}}\] is constant. 
\item Show that in all other cases, $\omega(z)$ depends on $z$ and is either a finite set, or a euclidean circle. 
\end{enumerate}
\end{exo}

\begin{exo}
Assume $f(z) = z^d$ with $d\geq 2$. 
\begin{enumerate}
\item Show that $\omega(z) = \{0\}$ if $z$ belongs to the unit disk $\mathbb{D}$ and $\omega(z)=\{\infty\}$ if $z\in \widehat{\mathbb{C}}\setminus\overline{\mathbb{D}}$. 
\item Show that 
\begin{enumerate}
\item $\omega(z)\subset S^1$ if $z$ belongs to the unit cirle $ S^1$, 
\item $\omega(z)$ can be  equal to $S^1$ and 
\item $\omega(z)$ can be a finite set of arbitrary cardinality. 
\end{enumerate}
\end{enumerate}
\end{exo}

\begin{exo}
Assume $f(z) = z^2-2$. 
\begin{enumerate}
\item Show that every $z\in \mathbb{C}\setminus\{0\}$ may be written $z = w+1/w$ for some $w\in \mathbb{C}$ with $|w| =  1$ if $z\in [-2,2]$ and $|w|>1$ if $z\in \mathbb{C}\setminus[-2,2]$. 
\item Show that 
\begin{enumerate}
\item $\omega(z) = \{\infty\}$ if $z\in \widehat{\mathbb{C}}\setminus [-2,2]$ and
\item $\omega(z)\subset [-2,2]$ if $z\in [-2,2]$. 
\end{enumerate}
\end{enumerate}
\end{exo}

\begin{exo}
Assume $f:\widehat{\mathbb{C}}\to \widehat{\mathbb{C}}$ is  a Tchebychev polynomial of degree $d$, i.e., such
that $\cos(d\theta)=f(\cos(\theta))$. Show that $\omega(z) = \{\infty\}$ if $z\in \widehat{\mathbb{C}}\setminus[-1,1]$ and 
$\omega(z)\subset [-1,1]$ if $z\in [-1,1]$. 
\end{exo}

Those examples are really atypical examples, since most cases exhibit {\em fractal} and {\em chaotic} behaviour, the analysis of which requires tools from complex
analysis, dynamical systems, topology, combinatorics, $\ldots$

\section{Fatou sets and Julia sets}
\label{sec:2}

Both Fatou's and Julia's work have as a major theme that the
Riemann sphere $\widehat{\mathbb{C}}$ breaks up sharply into two complementary
subsets: the open Fatou set $\mathcal{F}_f$ with orderly dynamics, and the
Julia set $\mathcal{J}_f$ on which the dynamical behaviour of $f$ is chaotic. In
practice, understanding the dynamics of a rational map has come
to mean understanding the topology of the Julia set, its geometry
if possible, and classifying the components of the Fatou set.

\subsection{Definition}
\label{subsec:2.1}

There are several possible definitions of the Fatou and Julia sets. We will adopt the one based on
normal families, which seems very natural to us: a normal
family of analytic maps is ``well-behaved''.  But the very fact
that it has become standard reflects the origins of holomorphic
dynamics in complex analysis, and may
explain why the subject has never quite entered the main stream of
dynamical systems.

\begin{definition}[Normal families]
A family $(f_\alpha:U\to\widehat{\mathbb{C}})_{\alpha\in {\cal A}}$ of holomorphic
maps on an open subset $U\subset \widehat{\mathbb{C}}$ is {\em normal} if 
from any sequence $(f_{\alpha_n})_{n\geq 0}$ one can extract a subsequence
converging uniformly (for the spherical metric in the range) on
compact subsets of $U$. \label{normalfamily} 
\end{definition}

A criterion that will enable us to deal with normality is due to Montel. 

\begin{theorem}[Montel] 
A family $(f_\alpha:U\to\widehat{\mathbb{C}})_{\alpha\in {\cal A}}$ of holomorphic
maps on an open subset $U\subset \widehat{\mathbb{C}}$ which
omits at least three values in $\widehat{\mathbb{C}}$ is normal.
\end{theorem}

In particular, a bounded sequence of holomorphic maps $(f_n:U\to \mathbb{C})_{n\in\mathbb{N}}$ is normal. 

\begin{definition}[Fatou sets and Julia sets]The {\em Fatou set} $\mathcal{F}_f$ of a rational map $f:\widehat{\mathbb{C}}\to \widehat{\mathbb{C}}$ is the largest open subset of $\widehat{\mathbb{C}}$ on which
the family of iterates $(f^{\circ n})_{n\in \mathbb{N}}$ is normal.
The {\em Julia set} $\mathcal{J}_f$ is the complement of the Fatou set. \label{Juliadef} 
\end{definition}

\begin{example}{Example} For $f(z)=z^d$, with $d\geq 2$ or $d\leq -2$, the Julia
set $\mathcal{J}_f$ is the unit circle 
\[S^1:=\{z\in\mathbb{C}~;~|z|=1\}.\]
Indeed,
the family of iterates $(f^{\circ n}:\widehat{\mathbb{C}}\setminus S^1
\to \widehat{\mathbb{C}})_{n\in \mathbb{N}}$ takes
its values in $\widehat{\mathbb{C}}\setminus S^1$, hence it omits more than three values
in $\widehat{\mathbb{C}}$. By Montel's theorem, it is normal and so, $\mathcal{J}_f\subseteq S^1$.
To prove that $S^1\subseteq \mathcal{J}_f$, one may argue that the sequence
$\bigl(f^{\circ(2n)}\bigr)_{n\in \mathbb{N}}$ converges uniformly on any compact
subset of $\mathbb{D}$ to the constant map $0$ and on any compact subset of
$\widehat{\mathbb{C}}\setminus \overline{\mathbb{D}}$ to the constant map $\infty$.
As any neighborhood of a point $z\in S^1$
contains points in $\mathbb{D}$ and points in $\widehat{\mathbb{C}}\setminus \overline{\mathbb{D}}$, there
is no neighborhood of $z$ on which the sequence $\bigl(f^{\circ n}\bigr)_{n\in \mathbb{N}}$ is
normal. 
\end{example}

\begin{figure}[htb]
\centerline{\includegraphics[height=4cm]{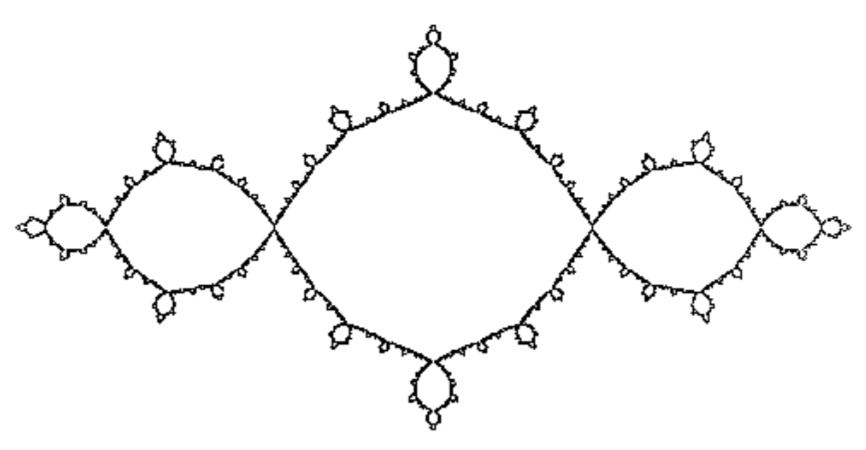}
} 
\caption{The Julia set of the quadratic polynomial
$f:\widehat{\mathbb{C}}\to\widehat{\mathbb{C}}$ defined by $f(z)=z^2-1$.\label{intro1}}
\end{figure}

\subsection{The polynomial case}
\label{subsec:2.2}
In the case of a polynomial,  there is a slightly
different way of understanding things. In this case, $\infty$ is a fixed point and this fixed point has no preimage other than itself. 
It is an exceptional point. 

\begin{definition}
Let $f:\widehat{\mathbb{C}}\to \widehat{\mathbb{C}}$ be a polynomial of degree $d\geq 2$. The {\em filled-in Julia set} of $f$ is the set $\mathcal{K}_f$
of points with bounded orbits:
\[\mathcal{K}_f = \Big\{z\in\mathbb{C}~\Big|~\bigl(f^{\circ n}(z)\bigr)_{n\in\mathbb{N}}~{\rm
  is~bounded}\Big\}.\]
\end{definition}
This set is not only very natural
for a dynamical system, but fits in very nicely with the general
theory.

\begin{proposition}\label{filledJulia}
The filled-in Julia set $\mathcal{K}_f$ is a non-empty compact set.
The Julia set is the topological boundary of $\mathcal{K}_f$.
\end{proposition}

\begin{proof} Since $f$ is of degree $d\geq 2$, $f(z)/z\to\infty$ as $z\to
\infty$, and so, there exists a constant $R>0$ such that
$\bigl|f(z)\bigr|>2|z|$ whenever $|z|>R$. Set $D_R:=\{z\in\mathbb{C}\mid |z|>R\}$. Then, the nested intersection of compact sets
\[\mathcal{K}_f= \bigcap_{n\in\mathbb{N}} f^{-n}\left(\overline {D}_R\right)\]
is a non-empty compact set.

Furthermore, define $U_R$ to be the open set
$U_R:=\widehat{\mathbb{C}}\setminus \overline {D}_R$. Then, the sequence
$(f^{\circ n})_{n\geq 0}$ converges uniformly to infinity on $U_R$.
Clearly, the same property holds on the open set
\[\Omega_\infty= \bigcup_{n\in\mathbb{N}} f^{-n}(U_R)\]
Observe that $\Omega_\infty=\widehat{\mathbb{C}}\setminus \mathcal{K}_f$ and
$\partial\Omega_\infty=\partial \mathcal{K}_f$.

We will first show that $\partial \mathcal{K}_f \subseteq \mathcal{J}_f$.
Choose a point $z_0\in \partial \Omega_\infty=\partial \mathcal{K}_f$
and assume that there is a
neighborhood $U$ of $z_0$ on which the family of iterates of $f$ is normal.
Then, $U$ intersects $\Omega_\infty$ and the sequence
$(f^{\circ n})_{n\geq 0}$ must converge to infinity on the whole
set $U$. This contradicts the fact that the sequence
$(f^{\circ n}(z_0))_{n\geq 0}$ is bounded.

We will now show that $\mathcal{J}_f\subseteq \partial \mathcal{K}_f$. As mentioned previously, the sequence $(f^{\circ n})$ converges locally uniformly to a map which is constant and equal to $\infty$ outside $\mathcal{K}_f$. Thus, the complement of $\mathcal{K}_f$ is contained in the Fatou set $\mathcal{F}_f$ and $\mathcal{J}_f\subseteq \mathcal{K}_f$. Similarly, on the interior of $\mathcal{K}_f$, the sequence $(f^{\circ n})$ is bounded, thus normal by Montel's theorem. As a consequence, the interior of $\mathcal{K}_f$ is contained in $\mathcal{F}_f$ and $\mathcal{J}_f\subseteq \partial \mathcal{K}_f$. 
\end{proof}


The reason why $\mathcal{K}_f$ is called a filled-in Julia set is the following. 

\begin{definition}
A compact subset $K \subset \mathbb{C}$ is called {\em full} if it is
connected and its complement is connected.
\end{definition}

\begin{proposition}\label{filledJuliafull}
Let $f:\widehat{\mathbb{C}}\to \widehat{\mathbb{C}}$ be a polynomial. Each connected component of $\mathcal{K}_f$ is full.
\end{proposition}

\begin{proof} Of course, any component $K_0$ is compact and connected. Any
component $U$ of the complement must satisfy $\partial U \subset
K_0$, and if the complement is not connected at least one such
component must be bounded. Then, by the maximum principle
\[\sup_U \bigl|f^{\circ n}(z)\bigr| = \sup_{\partial U} \bigl|f^{\circ n}(z)\bigr|\leq
\sup_{K_0} \bigl|f^{\circ n}(z)\bigr|.\] Hence, it is uniformly bounded, so $U
\subset K_0$. \end{proof}

\subsection{Invariance of Fatou sets and Julia sets}
\label{subsec:2.3}
\begin{proposition}
Both the Julia set and the Fatou set are completely invariant:
\[f(\mathcal{J}_f) = \mathcal{J}_f\,,\quad f^{-1}(\mathcal{J}_f) = \mathcal{J}_f \,,\quad
f(\mathcal{F}_f) = \mathcal{F}_f \,,\quad f^{-1}(\mathcal{F}_f) = \mathcal{F}_f\,.\]\label{invariance}
\end{proposition}

\begin{proof} As the Julia set is the complement of the Fatou set, it is enough
to prove the statements about the Fatou set. This is an immediate
consequence of the following lemma.
\begin{lemma}\label{normalinvariant}
Let $U$ be an open subset of $\widehat{\mathbb{C}}$, and let $\phi:U\to \widehat{\mathbb{C}}$ be holomorphic and
non-constant. Then the family $(g_\alpha:\phi(U)\to \widehat{\mathbb{C}})_{\alpha\in
\cal A}$ is normal if and only if $(f_\alpha:= g_\alpha\circ
\phi:U\to \widehat{\mathbb{C}})_{\alpha\in
\cal A}$ is normal.
\end{lemma}

\begin{remark}  Since $\phi$ is a non-constant analytic map, it is an open map.
Thus $\phi(U)$ is open and the family $g_\alpha$ is
defined on an open set.
\end{remark}

\begin{proof} ($\Rightarrow$) This is obvious. If the sequence $g_{\alpha_n}$
converges to $g$, then $f_{\alpha_n}:= g_{\alpha_n}\circ \phi$ converges to $g\circ
\phi$.

($\Leftarrow$) Assume that the sequence $f_{\alpha_k}:= g_{\alpha_k}\circ \phi$ is uniformly
convergent on every compact subset of $U$.
Let $K\subset \phi(U)$ be compact, and choose a cover of $U$ by
open subsets $U_i$ with compact closure.  Since $\phi$ is open, $\bigl\{\phi(U_i)\bigr\}$
is an open cover of $K$, and there exists a finite subcover
$\bigl\{\phi(U_{i_j})\bigr\}$.  The set
\[
K':= \phi^{-1}(K) \cap \bigcup_j \overline U_{i_j}
\]
is a compact subset of $U$ such that $\phi(K')=K$. Then,
\[
\sup_{z\in K} {\rm dist}_{\widehat{\mathbb{C}}}(g_{\alpha_n}(z), g_{\alpha_m}(z))= \sup_{w\in K'} {\rm dist}_{\widehat{\mathbb{C}}}(f_{\alpha_n}(w),
f_{\alpha_m}(w)) \to 0
\]
as $n,m\to \infty$. The lemma follows.
\end{proof}

To prove $\mathcal{F}_f=f^{-1}(\mathcal{F}_f)$, apply this lemma to the families
\[\Big\{f^{\circ (m+1)}:f^{-1}(\mathcal{F}_f)\to \widehat{\mathbb{C}}\Big\}\quad\hbox{and}\quad \Big\{f^{\circ m}:\mathcal{F}_f\to \widehat{\mathbb{C}}\Big\}\] and to the
holomorphic map $f:f^{-1}(\mathcal{F}_f)\to \mathcal{F}_f$. 
Applying $f$ to both sides of the equality  $\mathcal{F}_f=f^{-1}(\mathcal{F}_f)$, we deduce that $f(\mathcal{F}_f)=\mathcal{F}_f$. 
\end{proof}

\begin{proposition}
The Julia and Fatou sets of the $k$-th iterate $f^{\circ k}$ are equal to
the Julia and Fatou sets of $f$:
\[\mathcal{J}_{f^{\circ k}}=\mathcal{J}_f\quad\hbox{and}\quad
\mathcal{F}_{f^{\circ k}}=\mathcal{F}_f\]
\end{proposition}

\begin{proof}
It is enough to prove this proposition for the Fatou sets.
If the family of iterates $f^{\circ n}$ is normal on an open set $U$,
then the subfamily $f^{\circ (kn)}$ is also normal. Hence,
$\mathcal{F}_f\subseteq \mathcal{F}_{f^{\circ k}}$.

Now, assume that $z\in \mathcal{F}_{f^{\circ k}} $ and let $U$ be a
neighborhood of $z$ on which the family of iterates $f^{\circ (kn)}$
is normal. For any sequence $(n_i)_{i\in \mathbb{N}}$, there exists an
integer $N$ such that $n_i=N~\mod(k)$ for infinitely many $i$'s. Let
us extract this subsequence which, with an abuse of notations, we still denote by $n_i$. By
assumption, the sequence $f^{\circ (n_i-N)}:U\to\widehat{\mathbb{C}}$ is normal and we can
extract a subsequence converging to a map $g:U\to\widehat{\mathbb{C}}$. As
$f^{\circ n_i}=f^{\circ N}\circ f^{\circ (n_i-N)}$, we can extract a subsequence of the
sequence $f^{\circ n_i}:U\to \widehat{\mathbb{C}}$ converging to $f^{\circ N}\circ g:U\to \widehat{\mathbb{C}}$. This
shows that the family $\bigl(f^{\circ n}:U\to \widehat{\mathbb{C}}\bigr)_{n\in\mathbb{N}}$ is normal and that
$\mathcal{F}_{f^{\circ k}}\subseteq \mathcal{F}_f$. \end{proof}

\section{Conjugacy}
\label{sec:3}

In every part of mathematics, the {\em morphisms} are the maps that preserve the structure. 
Thus the morphisms of a dynamical system $f:X \to X$ to a dynamical system $g:Y \to Y$ are 
the maps sending sequences of iterates of $f$ to sequences of iterates of $g$. 
This is just what maps $\phi:X \to Y$ such that $\phi\circ f = g \circ \phi$ do; 
such maps are called {\em semi-conjugacies} between $f$ and $g$. 
Indeed, we see by induction that then $\phi\circ f^{\circ n}= g^{\circ n} \circ \phi$ for all $n\ge 1$:  
by hypothesis it is true for $n=1$, and if it is true for $n-1$, then
\[
\phi\circ f^{\circ n}= \phi\circ f\circ  f^{\circ (n-1)}= g \circ \phi \circ f^{\circ (n-1)}= 
g \circ g^{\circ (n-1)}\circ \phi= g^{\circ n}\circ \phi.
\]
 When $\phi$ is invertible, $\phi\circ f = g \circ \phi$ can be rewritten $f=\phi^{-1}\circ g\circ \phi$, so $f$ and $g$ are {\em conjugate}: conjugacies are the isomorphisms of dynamical systems. When studying dynamical systems, we constantly try to construct conjugacies or semi-conjugacies to simpler dynamical systems, linear ones in particular. 

\begin{definition}[Conjugate rational maps]
Two rational maps $f$ and $g$ are {\em topologically
(respectively analytically) conjugate} if there exists a homeomorphism $\phi:\widehat{\mathbb{C}}\to \widehat{\mathbb{C}}$
(respectively an analytic isomorphism) such that $\phi\circ f=g\circ \phi$.
\end{definition}

One should think of $\phi$ as a change of coordinates on $\widehat{\mathbb{C}}$. The only
analytic isomorphisms of $\widehat{\mathbb{C}}$ are the M\"obius transformations
$z\mapsto (az+b)/(cz+d)$ and the only isomorphisms of $\mathbb{C}$ are the
affine maps $z\mapsto az+b$. Being a zero or a pole of a rational
map depends on the coordinates, whereas being a fixed point or a
critical point of a rational map does not depend on the choice of
coordinates. 

\begin{proposition}\label{juliaconjhomeo}
If $f$ and $g$ are conjugate (topologically or analytically) by $\phi:\widehat{\mathbb{C}}\to \widehat{\mathbb{C}}$,
then $\phi(\mathcal{F}_f)=\mathcal{F}_g$ and $\phi(\mathcal{J}_f)=\mathcal{J}_g$.
\end{proposition}

\begin{proof}
The family $(f^{\circ n}:U\to \widehat{\mathbb{C}})_{n\in\mathbb{N}}$ is normal if and only if the family
$(\phi\circ f^{\circ n}\circ \phi^{-1}: \phi(U)\to \widehat{\mathbb{C}})_{n\in
\mathbb{N}}$ is normal.
\end{proof}

\begin{example}{Example} We will see below that the Newton's methods 
\[N_P:z\mapsto 
z-\frac{P(z)}{P'(z)}\quad \hbox{and} \quad N_Q: z\mapsto  z - \frac{Q(z)}{Q'(z)}\] 
are analytically
conjugate as soon as $Q(z) = \lambda P(\alpha z+\beta)$. 
\end{example}

\begin{exo}
Let $a\neq b$ be complex numbers and set $P(z) :=(z-a)(z-b)$. 
\begin{enumerate}
\item Show that the Newton's method $N_P$ is conjugate to  $f:w\mapsto w^2$ 
via the isomorphism $z\mapsto w := (z-a)/(z-b)$. 
\item Determine the Julia set of $N_P$ and the set $\omega(z)$ for $z\in \widehat{\mathbb{C}}\setminus \mathcal{J}_{N_P}$. 
\end{enumerate}
\end{exo}

\section{Periodic point and critical points}
\label{sec:4}
Periodic points and critical points are important objects in holomorphic dynamics as we shall see in this section.

\subsection{Periodic points}
\label{subsec:4.1}

\begin{definition}[Multiplier at a fixed point]
A {\em fixed point} of $f$ is a point $\alpha$ such that $f(\alpha)=\alpha$. 
The derivative of $f$ at $\alpha$ is an endomorphism of the tangent line ${\rm D}_\alpha f: T_\alpha\widehat{\mathbb{C}}\to T_\alpha\widehat{\mathbb{C}}$, thus a multiplication by a number $\lambda\in \mathbb{C}$. The eigenvalue  $\lambda$ of the endomorphism ${\rm D}_\alpha f: T_\alpha\widehat{\mathbb{C}}\to T_\alpha\widehat{\mathbb{C}}$ is called the {\em multiplier} of $f$ at $\alpha$. The fixed point is called
\begin{itemize}
\item[$\bullet$]{\em superattracting} if $\lambda=0$;
\item[$\bullet$]{\em attracting} if $0<|\lambda|<1$;
\item[$\bullet$]{\em repelling} if $|\lambda|>1$;
\item[$\bullet$]{\em indifferent} if $|\lambda|=1$, and more precisely
\begin{itemize}
\item {\em parabolic} if $\lambda$ is a root of 1,
\item {\em elliptic} otherwise, that is if $\lambda=e^{2\pi i \theta}$ with $\theta\in\mathbb{R}\setminus\mathbb{Q}$.
\end{itemize}
\end{itemize}
\end{definition}

\begin{exo} Let $\alpha$ be a fixed point of a holomorphic map $f$. 
\begin{enumerate}
\item 
Show that  if $\alpha\neq \infty$, then the multiplier of $f$ at $\alpha$ is 
$\lambda = f'(\alpha)$.
\item 
Show that if $\alpha=\infty$ and if $f(z)\sim az$ as $z\to \infty$, then the multiplier is $\lambda = 1/a$, whereas if $f(z)\sim az^k$ with $k\geq 2$, the multiplier is $\lambda=0$. 
\item Show that if $f$ is a polynomial of degree $d\geq 2$, then $\infty$ is a superattracting fixed point of $f$. 
\end{enumerate}
\end{exo}

The multiplier at a fixed point is invariant under analytic conjugacy. In other words,  if  $\phi:(\widehat{\mathbb{C}},\alpha)\to (\widehat{\mathbb{C}},\beta)$ is a local isomorphism conjugating 
$f:(\widehat{\mathbb{C}},\alpha)\to (\widehat{\mathbb{C}},\alpha)$ to $g:(\widehat{\mathbb{C}},\beta)\to (\widehat{\mathbb{C}},\beta)$, then the multiplier of $f$ at $\alpha$ is equal to the multiplier of $g$ at $\beta$. Indeed, the derivative of $\phi\circ f = g\circ \phi$ at $\alpha$ is 
\[{\rm D}_\alpha \phi\circ {\rm D}_{\alpha} f = {\rm D}_\beta g \circ {\rm D}_\alpha \phi.\]
In other words, ${\rm D}_\alpha \phi$ conjugates ${\rm D}_\alpha f:T_\alpha \widehat{\mathbb{C}}\to T_\alpha\widehat{\mathbb{C}}$ to ${\rm D}_\beta g:T_\beta \widehat{\mathbb{C}}\to T_\beta\widehat{\mathbb{C}}$ and the two endomorphisms have the same eigenvalue.

\begin{definition}[Periodic points]
A {\em periodic point} of $f$ is a fixed point of $f^{\circ p}$ for some 
$p\geq 1$. The smallest such integer $p\geq 1$ is called the period of $\alpha$.
In this case, we say that $\left<\alpha,f(\alpha),\ldots,f^{\circ (p-1)}(\alpha)\right>$
is a {\em cycle of $f$}. 
\end{definition}

\begin{proposition}
If $\left<\alpha_1,\alpha_2,\ldots,\alpha_{p}\right>$ is a cycle, then the multiplier of $f^{\circ p}$ at all points of the cycle is the same.
\end{proposition}

\begin{proof}
According to the Chain Rule, the derivative of $f^{\circ p}$ at any point of the cycle is the composition of the derivatives of $f$, as linear transformations, along the cycle. 
If the derivative of $f$ vanishes at a point of the cycle, then the multiplier is $0$ at all points of the cycle. 
If the derivative never vanishes, then $f$ is locally invertible at each point of the cycle and $f$ conjugates $f^{\circ p}$ at $\alpha_i$ to $f^{\circ p}$ at $\alpha_{i+1}$. Thus, the multiplier of $f^{\circ p}$ at $\alpha_i$ is equal to the multiplier of $f^{\circ p}$ at $\alpha_{i+1}$. 
\end{proof}

\begin{definition}[Multiplier along a cycle]
The {\em multiplier} of $f$ along a cycle of period $p$ is the multiplier of $f^{\circ p}$ at any point of the cycle. The cycle is {\em superattracting (respectively attracting, indifferent or repelling)} for $f$ if the points of the cycle are superattracting (respectively attracting, indifferent or repelling) fixed points of $f^{\circ p}$. 
\end{definition}

If $\left<\alpha_1,\alpha_2,\ldots,\alpha_{p}\right>$ is a cycle of period $p$ avoiding $\infty$, then its multiplier is given by the product
\[\lambda = (f^{\circ p})'(\alpha_1) = \prod_{i=1}^p f'(\alpha_i).\]

\begin{proposition}\label{wherefixedpoints}
Let $f:\widehat{\mathbb{C}}\to \widehat{\mathbb{C}}$ be a rational map, and let $\alpha$ be
a periodic point of $f$. If $\alpha$ is superattracting or attracting,
then $\alpha\in \mathcal{F}_f$.
If $\alpha$ is repelling, then $\alpha\in \mathcal{J}_f$.
\end{proposition}

\begin{proof} Replacing $f$ by
$f^{\circ p}$ if necessary, we may assume that $\alpha$ is a fixed
point of $f$. Conjugating by a M\"obius transformation if necessary, we may
assume that $\alpha=0$. If $\alpha$ is (super)attracting, i.e.,
$|f'(\alpha)|<1$, then there exists a bounded neighborhood $U$ of $\alpha$
which is mapped into itself. In particular, for all $n\in \mathbb{N}$,
$f^{\circ n}(U)\subset U$. Hence, the family $(f^{\circ
n}:U\to\widehat{\mathbb{C}})_{n\in\mathbb{N}}$ is normal, which implies that $\alpha\in
\mathcal{F}_f$. If $\alpha$ is repelling, then 
\[(f^{\circ n})'(\alpha) = \bigl[f'(\alpha)\bigr]^n\underset{n\to +\infty}\longrightarrow \infty.\]
So, the family $(f^{\circ n})_{n\in \mathbb{N}}$ cannot be normal in a
neighborhood of $\alpha$. \end{proof}

A cycle is superattracting if and only if it contains a point at which the derivative of $f$ vanishes, i.e., a {\em critical point} of $f$. 

\subsection{Critical points and critical values}
\label{subsec:4.2}

We will see that the dynamical properties of a rational map $f:\widehat{\mathbb{C}}\to \widehat{\mathbb{C}}$
are strongly related to the dynamical behaviour of the critical points of
$f$.  For example, the Julia set of a polynomial $f$ is
connected if and only if the orbit of every critical point of $f$ is
bounded.

\begin{definition}[Critical points and critical values]
Let $U$ and $V$ be two Riemann surfaces and $f:(U,x)\to (V,y)$ be an 
analytic map which is not locally constant at $x$. 
The point $x$ is a {\em critical point} if the derivative ${\rm D}_xf:T_xU\to T_yV$ is zero. In that case, $y$ is a {\em critical value}. 
\end{definition}

More precisely, in local coordinates $z$ vanishing at $x$ and $w$ vanishing at $y=f(x)$, the expression
of $f$ is of the form
\[z\mapsto aw^k+{O}(w^{k+1})\quad \text{for some integer }k\geq 1~\text{and some complex number } a\neq 0.\]
The integer $k$ does not depend on the choice of coordinates and is called the {\em local degree} of $f$ at $x$. 
The point $x$ is a critical point if and only if $k\geq 2$. In that case, $x$ is called a {\em critical point of $f$ of multiplicity $k-1$}.

\begin{example}{Example} The map $z\mapsto z^3$ has a critical
point at $0$.  The local degree of the map at $0$ is $3$, and
the multiplicity of this critical point is $2$.
\end{example}

We shall use the following result which is a consequence of the Riemann-Hurwitz Formula (which will not be proved in those notes). 

\begin{definition}[Euler characteristic]
If $U$ is an open subset of $\widehat{\mathbb{C}}$ with $m$ boundary components, the {\em Euler characteristic of $U$} is 
\[\chi(U) := 2-m.\]
\end{definition}

\begin{theorem}\label{theo:RH}
Let $f:\widehat{\mathbb{C}} \to \widehat{\mathbb{C}}$ be a rational map of degree $d$, $V\subseteq \widehat{\mathbb{C}}$ be an open set and $U:=f^{-1}(V)$.  Then, 
\[\chi(U) = d\cdot \chi(V) - N\]
where $N$ is the number of critical points of $f$ in $U$, counting multiplicities. 
\end{theorem}

\begin{exo}
Let $f:\widehat{\mathbb{C}}\to \widehat{\mathbb{C}}$ be a rational map of degree $d$. 
\begin{enumerate}
\item Prove that $f$ has $2d-2$ critical points counting multiplicities. 
\item Prove that $f$ has at least $2$ distinct critical values. 
\item Prove that if $f$ is a polynomial, then $\infty$ is a critical point of multiplicity $d-1$ and that there are $d-1$ critical points in $\mathbb{C}$ counting multiplicities. 
\end{enumerate}
\end{exo}

\section{Description of the Julia set}
\label{sec:5}

\subsection{Topology of the Julia set}
\label{subsec:5.1}

\begin{proposition}
The Julia set $\mathcal{J}_f$ is non empty and compact.
\end{proposition}

\begin{proof} If the Julia set were empty, there would be a subsequence
of $f^{\circ m}$ converging uniformly on $\widehat{\mathbb{C}}$; let $f_0$ be
the limit. The rational map $f_0$ must have some finite degree. Since
the degree is an invariant of homotopy, the approximating iterates
$f^{\circ m_i}$ must have the same degree.  But the degree of
$f^{\circ m_i}$ is $d^{m_i}$, which leads to a contradiction.

Evidently the Fatou set $\mathcal{F}_f$ is open, hence the Julia set
$\mathcal{J}_f$ is closed. Since $\widehat{\mathbb{C}}$ is compact, the Julia set $\mathcal{J}_f$ is
also compact. 
\end{proof}

\begin{proposition}\label{perfect}
Either the Julia set $\mathcal{J}_f$ has empty interior, or it
is the entire Riemann sphere.
\end{proposition}

\begin{proof} 
Suppose $U \subset \mathcal{J}_f$ is an open subset of $\widehat{\mathbb{C}}$.  The family
$\{f^{\circ m}:U\to\widehat{\mathbb{C}}\}$ must avoid $\mathcal{F}_f$, which is open
and will contain more than two points if it is not empty. But this
would make $\{f^{\circ m}:U\to\widehat{\mathbb{C}}\}$ normal, a contradiction. Thus if the
interior of $\mathcal{J}_f$ is non empty, $\mathcal{J}_f$ is the entire Riemann sphere.
\end{proof}

We have seen that the Julia set $\mathcal{J}_f$ is completely invariant. We will now
determine whether there are any other completely invariant
closed sets than $\mathcal{J}_f$. For this purpose, we will use the Riemann
Hurwitz formula.

\begin{proposition}\label{totinvfinite}
Let $E$ be a completely invariant closed set. If $E$ is finite, then $E$ contains at most two points.
\end{proposition}

\begin{proof} If $E$ is finite, then its complement $U$ is connected and
completely invariant. Its Euler characteristic $\chi(U)=2-{\rm card} (E)$
is finite. Let us denote by $N$ the number of critical points of $f$
in $U$ counted with multiplicity. The Riemann-Hurwitz formula
applied to $f:U\to U$ gives
\[\chi(U)=d\cdot \chi(U)-N.\]
Thus
\[\chi(U)=\frac{N}{d-1}\geq 0,\]
which leads to ${\rm card} (E)\leq 2$. \end{proof}

A union of completely invariant sets is still completely invariant, therefore 
the following definition makes sense.
\begin{definition}
The exceptional set $E_f$ is the largest finite completely invariant set.
\end{definition}

\begin{remark} The grand-orbit of a point is a totally invariant set. If it
is a finite set, then, by definition of the exceptional set $E_f$, we
have $z\in E_f$.
\end{remark}

Observe that when two rational maps $f:\widehat{\mathbb{C}}\to \widehat{\mathbb{C}}$ and $g:\widehat{\mathbb{C}}\to \widehat{\mathbb{C}}$ are
conjugate by a homeomorphism $\phi$, then the homeomorphism $\phi$ sends
the forward (resp. backward) orbit of a point $z\in \widehat{\mathbb{C}}$ under the action of
$f$
to the forward (resp. backward) orbit of $\phi(z)$ under the action of $g$.
Hence, $f$-invariant sets are mapped by $\phi$ to $g$-invariant
sets. Thus, the largest finite completely $f$-invariant set is mapped by
$\phi$ to the largest finite completely $g$-invariant set. In other words,
when it is not empty, the exceptional set $E_f$ is
mapped by $\phi$ to the exceptional set $E_g$.

The exceptional set $E_f$ is a minor irritant, which will
come back to annoy us on several occasions. The following
description should remove any doubts about its lack of
significance.

\begin{proposition}\label{exceptionalset}
Up to conjugacy, the exceptional set $E_f$ is not empty in precisely two cases:
\begin{itemize}
\item when $f(z)=z^d$ with $|d| \ge 2$, $E_f=\{0,\infty\}$ and

\item when $f$ is a polynomial but is not conjugate to $z\mapsto z^d$, $E_f=\{\infty\}$.
\end{itemize}
\end{proposition}

\begin{proof} It is clear from the construction that $E_f$ is completely
invariant; so if it has only one point, this point must be fixed,
and if it consists of two points, these must be fixed or exchanged by
$f$. We will examine these three cases one at a time.

If $E_f$ has just one point $a$, move it to $\infty$, for example by
making the
change of coordinates $w=1/(z-a)$. Then $f = P/Q$ is
a rational map such that infinity is its only inverse image. Hence,
the polynomial $Q$ has no roots, it is constant,  and $f$
is a polynomial.

If $E_f$ has two points $a$ and $b$, both fixed, then we can put one at
$\infty$ and the other at $0$ (make the change of coordinates
$w=(z-b)/(z-a)$). We see that $f$ is a polynomial and that $0$
is the only root of the polynomial equation $f(w)=0$, so $f(w) =
\lambda w^d$ for some $\lambda$. Choose a number $\nu$ such that
$\nu^{d-1}=\lambda$, and make the change of coordinates $z=\nu w$. Then,
it is easy to see that the expression of $f$ becomes $z \mapsto z^d$.

If $E_f$ has two points exchanged by $f$, put one
at $\infty$ and the other at $0$, and write $f=P/Q$. Since
$\infty$ is the only inverse image of $0$, we see that $P$ has no
roots, so it is a constant, which we may take to be $1$. Since $0$
is the only inverse image of $\infty$, we see that $Q(w) = \lambda w^d$ for
some $\lambda \neq 0$, so that $f(w)= 1/(\lambda w^d)$. We can do
a change of coordinate $z=\nu w$, where $\nu^{d+1}=\lambda$. Again it is easy
to check that the expression of $f$ becomes $z \mapsto z^{-d}$.
\end{proof}

\begin{proposition}\label{exceptsetinFatou}
The exceptional set $E_f$ is always contained in the Fatou
set $\mathcal{F}_f$.
\end{proposition}

\begin{proof}
In all three cases, $E_f$ is a union of
superattracting cycles. Proposition \ref{wherefixedpoints} shows that
such cycles are always contained in
the Fatou set $\mathcal{F}_f$.
\end{proof}

We can now give a new characterization of the Julia set.

\begin{proposition}\label{juliaininvariant}
The Julia set of a rational map is the smallest completely invariant closed set
containing at least three points.
\end{proposition}

\begin{proof}
The Julia set is a completely invariant closed set. It contains at least three
points, since otherwise it would be
contained in the exceptional set $E_f\subset \mathcal{F}_f$.
Furthermore, if $E$ is a completely invariant closed set containing at
least $3$ points, the complement $\Omega=\widehat{\mathbb{C}}\setminus E$ of $E$ is a completely
invariant open set omitting more than three points. Hence, by Montel's
theorem, the family of iterates $f^{\circ n}$ is normal on $\Omega$.
Thus $\Omega\subset \mathcal{F}_f$ and $\mathcal{J}_f\subset E$.
\end{proof}

The following proposition asserts that if $z$ does not belong to the
exceptional set $E_f$, then the closure of the backward orbit of $z$,
\[{\cal O}^-(z):=\bigcup_{m\geq 0}f^{-m}\{z\},\]
contains the Julia set $\mathcal{J}_f$.

\begin{proposition}\label{invimdense}
For any $z \notin E_f$, we have
\[\mathcal{J}_f \subseteq \overline{{\cal O}^-(z)}.\]
\end{proposition}

\begin{proof} Take a point $z_0\in \mathcal{J}_f$ and let $U$ be an arbitrary
neighborhood of $z_0$. We must show that the set
$\Omega=\bigcup_{n\in \mathbb{N}} f^{\circ n}(U)$
contains $z$. We will show that its complement $E=\widehat{\mathbb{C}}\setminus \Omega$
is contained in the exceptional set $E_f$.
Indeed, $\Omega$ is a forward invariant open
set which omits at most two points, since otherwise the family
$f^{\circ n}$ would be normal on $U$.
Thus, $E=\widehat{\mathbb{C}}\setminus \Omega$ is a backward invariant set. Since $E$ is
finite, it is permuted by $f$, so it is forward invariant. Hence,
by definition of the exceptional set, it is contained in $E_f$.
\end{proof}

\begin{proposition}\label{perfect2}
The Julia set $\mathcal{J}_f$ is perfect. 
\end{proposition}

\begin{proof} Since the Julias set is closed, it suffices to prove that it has no isolated points. The Julia set $\mathcal{J}_f$ is not finite. Thus, it contains a
non-isolated point $z_0$. Since the backward orbit of $z_0$ is
dense in $\mathcal{J}_f$, any point in $\mathcal{J}_f$ is non-isolated. This proves
that $\mathcal{J}_f$ is perfect, hence uncountable.
\end{proof}

\subsection{Julia sets and periodic points}
\label{subsec:5.2}

Fatou and Julia both proved that the Julia set is the closure of the
set of repelling periodic points. This ``dynamical'' definition of the Julia
set is probably more natural to specialists of dynamical systems, and
relates holomorphic dynamics to ``Axiom A attractors'' and hyperbolic
dynamics.
Fatou's and Julia's proofs are different but both involve results
which will be proved later. Fatou's proof is based on the fact that
there are only
finitely many non-repelling cycles. 
Julia's proof is based on the existence of a repelling cycle.

In these notes, we will only prove the following weaker result which follows directly from Montel's
Theorem. This result is the first step in Fatou's proof.
The density of repelling cycles in the Julia set then follows from
the finiteness of non-repelling cycles.

\begin{proposition}\label{perptdense} The closure of the set of
periodic points contains the Julia set.
\end{proposition}

\begin{proof} Let $z_0 \in \mathcal{J}_f$ be arbitrary. Since $\mathcal{J}_f$ is perfect, we may assume
that $z_0$ is not a critical
value of $f^{\circ 2}$, so there exist a neighborhood $U$ of $z_0$ and three
branches $g_1,g_2,g_3$ of $f^{-2}$ with mutually disjoint images.
We are considering the second iterate because if $d=2$, there would
only be two branches of $f^{-1}$.

Now consider the family of maps
\[ h_m(z):= \frac {f^{\circ m}(z)-g_1(z)}{f^{\circ m}(z)-g_2(z)}\left/
\frac{g_3(z)-g_1(z)}{g_3(z)-g_2(z)}\right..
\]
If $f^{\circ m}(z) \ne g_i(z), i=1,2,3$ for all $z$ in some neighborhood
$U$ of $z$, then the maps $h_m(z)$ are never $0,1$ or
$\infty$ in $U$, hence they form a normal family.  But then the family
$f^{\circ m}$ is also normal, since we can express $f^{\circ m}$ in terms of
$h_m$:
\[
f^{\circ m} = \frac {g_1(g_3-g_2)-g_2(g_3-g_1)h_m}{(g_3-g_2)- (g_3-g_1)h_m}.
\]
This is a contradiction, so in every neighborhood of $z$ there are
roots of $f^{\circ m}(z) = g_i(z)$ for some $i$ and some $m$, i.e.,
roots of $f^{\circ(m+2)}(z) =z$. \end{proof}

\subsection{Connectivity of Julia sets of polynomials}
\label{subsec:5.3}

Here is a first relation between the behaviour of critical orbits and the topology of the Julia set. 

\begin{theorem}[Fatou]
The filled-in Julia set $\mathcal{K}_f$ of a polynomial $f$  is connected if and only if all critical points of
$f$ belong to $\mathcal{K}_f$.\label{filledjulconnected}
\end{theorem}

\begin{remark}
Note that since the Julia set $\mathcal{J}_f$ is the boundary of $\mathcal{K}_f$,
$\mathcal{J}_f$ is connected if and
only if $\mathcal{K}_f$ is connected. 
\end{remark}

\begin{proof} Since $f$ is a polynomial of degree $d\ge 2$, we have that $f(z)/z$ tends to infinity as $z$ tends to infinity, hence there exists $R>0$ so that $\bigl|f(z)\bigr|\geq 2|z|$ for $|z|>R$.
Set $U_0:=\widehat{\mathbb{C}}\setminus \overline {D}_R$ and define
inductively $U_{n+1}:= f^{-1}(U_n)$. By the Riemann-Hurwitz
formula, we have
\[
\chi(U_{n+1})=d\chi(U_n)-(d-1),
\]
where the number $d-1$ is the multiplicity of  the critical point at
$\infty$. Since $U_0$ is simply connected, we have $\chi(U_0)=1$,
and by induction, we see that when all the critical points of $f$
belong to $\mathcal{K}_f$, then $\chi(U_n)=1$ for all $n\geq 0$. It follows
that
\[\Omega_\infty:= \bigcup_{n\geq 0}U_n\]
is simply connected and its complement $\mathcal{K}_f$ is connected.

Conversely, let $\Omega_\infty\subset {\widehat{\mathbb{C}}}$
be the connected component of the Fatou set $\mathcal{F}_f$ which
contains $\infty$. Since the only
inverse image of $\infty$ is $\infty$ itself,
$f^{-1}(\Omega_{\infty})=\Omega_\infty$. Applying the Riemann-Hurwitz
formula to $f:\Omega_\infty \to \Omega_\infty$, we get
\[\chi(\Omega_\infty) = d \chi(\Omega_\infty) - (d-1) - k,\]
where the number $d-1$ is contributed by the critical point at
$\infty$, and the number $k$ is the number of other critical points
in $\Omega_\infty$ (i.e., not in $\mathcal{K}_f$) counted with multiplicity.
If $\chi(\Omega_\infty)$ is finite, then $\chi(\Omega_\infty) =
1+k/(d-1)$.  But the Euler characteristic of a connected plane
domain is at most 1.  So $\chi(\Omega_\infty)=1$ if no critical
point is attracted to $\infty$; otherwise $\chi(\Omega_\infty)$ is
infinite.
\end{proof}


\section{Description of the Fatou set}
\label{sec:6}
We now return to the case where $f$ is a rational map, not
necessarily a polynomial. Our goal is to give a classification of periodic Fatou components. 

\begin{definition}
A map $f:U\to V$ is a {\em proper map} if for any compact set $K\subseteq V$, 
the preimage $f^{-1}(K)$ is compact. 
\end{definition}

\begin{proposition}\label{stableproper}
If $U$ is a connected component of $\mathcal{F}_f$, then $f(U)$ is also a
connected component of $\mathcal{F}_f$ and $f:U\to f(U)$ is a proper
map.
\end{proposition}

\begin{proof} The rational map $f:\widehat{\mathbb{C}}\to \widehat{\mathbb{C}}$ is a proper map.
Hence, for any open set $V\subset\widehat{\mathbb{C}}$, $f:f^{-1}(V)\to V$ is a proper
map. In particular, this holds for $U=\mathcal{F}_f$. But since
$f^{-1}(\mathcal{F}_f)=\mathcal{F}_f$ by Proposition \ref{invariance}, we get that
$f:\mathcal{F}_f\to \mathcal{F}_f$ is a proper map. Then the proposition
follows, since the restriction of a proper map to a connected
component of the domain is still proper.
\end{proof}

\begin{definition}
A connected component $U$ of $\mathcal{F}_f$ is called a {\em Fatou component of $f$}. 
\end{definition}

\begin{definition}
A Fatou component is  {\em periodic} if
there exists an integer $k\geq 1$ such that
\[f^{\circ k}(U)=U.\]
\end{definition}

Note that if $U$ is periodic of period $k$ for $f$, then it is invariant
for $f^{\circ k}$.

\begin{definition}
An invariant Fatou component $U$ of $f$ is called
\begin{itemize}
\item[$\bullet$]{a {\em (super)attracting domain} if there is a
(super)attracting fixed point $\alpha\in U$ ($f(\alpha)=\alpha$ and
$0\le|f'(\alpha)|<1$) and the sequence $f^{\circ n}$ converges uniformly to
$\alpha$ on every compact subset of $U$;}
\item[$\bullet$]{a {\em parabolic domain} if there is a fixed point $\alpha\in
\partial U$ with $f'(\alpha)=1$, and the sequence $f^{\circ n}$
converges uniformly to
$\alpha$ on every compact subset of $U$;}
\item[$\bullet$]{a {\em Siegel disk} if it is simply connected, and if there
exists
an isomorphism $h:U\to \mathbb{D}$ such that
\[h\circ f\circ h^{-1}(z)=e^{2i\pi\theta}z\]
with $\theta\in \mathbb{R}\setminus \mathbb{Q}$;}
\item[$\bullet$]{a {\em Herman ring} if it is doubly connected, and if there
exists a radius $r$ and an isomorphism
\[h:U\to {\cal A}_r:=\{z\in \mathbb{C}~;~r<|z|<1\},\]
such that
\[h\circ f\circ h^{-1}(z)=e^{2i\pi\theta }z\]
with $\theta\in \mathbb{R}\setminus \mathbb{Q}$.}
\end{itemize}
\end{definition}

\begin{remark}
Clearly, the possibilities are mutually exclusive. Theorem
\ref{classificationFatou} below asserts that these are the only
possibilities and it is known that they all occur.
The existence of
Siegel disks was proved in 1942 by Siegel, and the existence of Herman
rings was proved by Herman in 1981.

Figure \ref{rotdomains} shows the Julia set of a polynomial having a Siegel disk
and the Julia set of a rational map having a Herman ring.
\begin{figure}[htb]
\centerline{
\begin{picture}(0,0)%
\includegraphics[height=6.8cm]{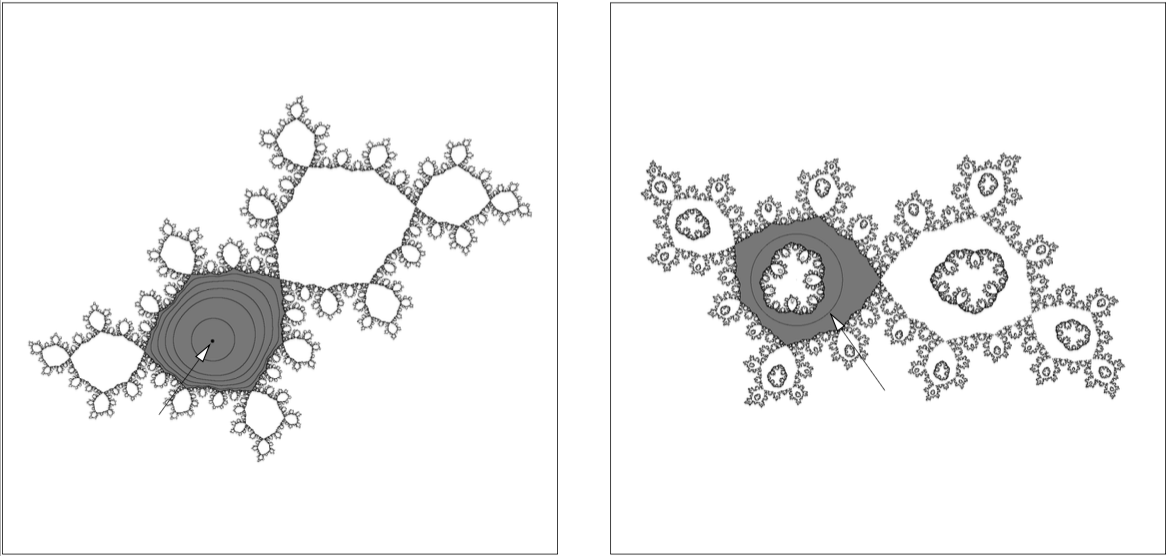}%
\end{picture}%
\setlength{\unitlength}{1657sp}%
\begin{picture}(16122,7676)(739,-7350)
\put(2701,-5731){$0$}
\put(7516,-1636){$\mathcal{J}_P$}
\put(12871,-5371){$S^1$}
\put(15211,-5416){$\mathcal{J}_f$}
\end{picture}
}
\caption{Left: the polynomial $P:z\mapsto e^{i\pi(\sqrt5 -1)} z+ z^2$ has a
Siegel disk (colored grey). We have drawn the
orbits of some points in the Siegel disk. Each orbit accumulates on a
$\mathbb{R}$-analytic circle.
Right: if $t\in \mathbb{R}/\mathbb{Z}$ is chosen carefully enough, the rational map
$f(z)=e^{2i\pi t} z^2(z-4)/(1-4z)$ has a Herman ring. It leaves the circle
$S^1$ invariant and is conjugate to an irrational rotation on $S^1$. The
picture is drawn for $t=0.61517321588\ldots$
}\label{rotdomains}
\end{figure}
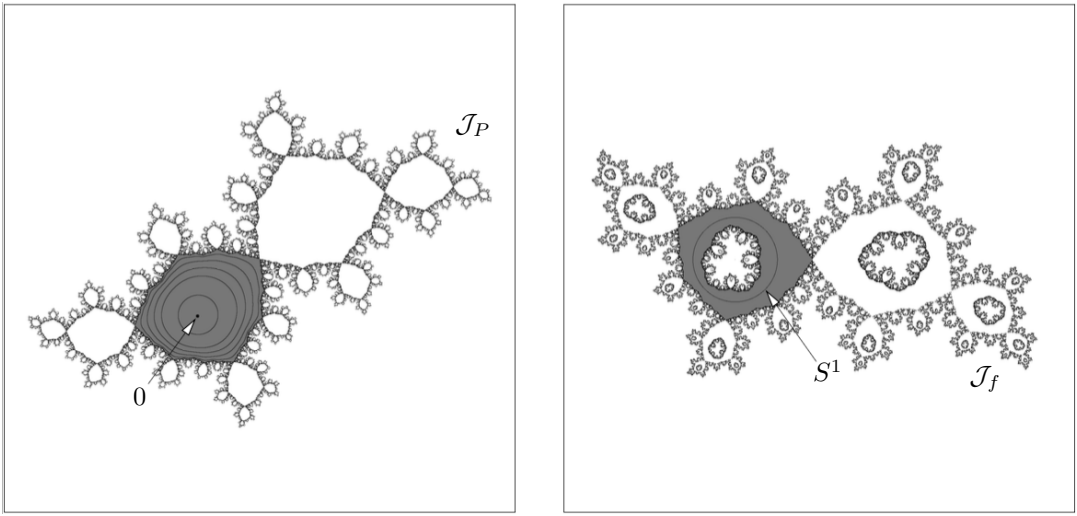
\end{remark}

\begin{proposition}
If $U$ is an invariant Fatou component of a rational map $f$ of degree $d\ge 2$, then
it is either simply connected, or doubly connected or
its complement has infinitely many connected components. It is doubly
connected if and only if it is a Herman ring.
\end{proposition}

\begin{proof} We can apply the Riemann-Hurwitz formula to $f:U\to U$. It gives
\[\chi(U)=d\chi(U)-n\]
where $d$ is the degree of $f:U\to U$ and $n$ is the number of
critical points of $f$ in $U$, counting multiplicities. If $\chi(U)$
is finite, it follows that $\chi(U)=n/(d-1)$ is non negative, and
since $U$ is not the entire Riemann sphere, we have $\chi(U)=0$ or
$1$. Moreover, if $\chi(U)=0$, i.e., if $U$ is doubly connected,
then $n=0$ and $f:U\to U$ is a covering map. In addition, the degree
of $f:U\to U$ is $1$ since it preserves the modulus of $U$. Thus,
$f:U\to U$ is an isomorphism. This isomorphism cannot be of finite order, since otherwise an iterate of $f$ would have infinitely many fixed points. Therefore $U$ is a Herman ring. \end{proof}

\begin{remark}
If $f$ is a polynomial and $U$ is a bounded Fatou component, then thanks to the maximum principle $U$ is simply connected. In particular, there are no Herman rings. 
\end{remark}

The following result is due to Fatou. We will not give the detailed proof. 

\begin{theorem}[Classification of invariant Fatou
components] Let $f:\widehat{\mathbb{C}}\to \widehat{\mathbb{C}}$ be a rational function of degree $d\ge 2$. An invariant Fatou component
$U$ of $f$ is a (super)attracting domain, a parabolic domain,
a Siegel disk or a Herman ring.\label{classificationFatou} 
\end{theorem}

\begin{proof}[Sketch of the proof] The classification starts by studying what are the possible limit values of the sequence of iterates $(f^{\circ n}:U\to U)$.

\noindent{\bf Case 1.} There is a non constant limit value. 

\noindent{\bf Case 1.1.} There is a sequence $(n_k)$ such that $f^{\circ n_k}\to {\rm Id}$. This is a first place where we do not provide the details. One proves that 
\begin{itemize}
\item either $f$ has finite order (which is not the case since if $f^{\circ p} = {\rm Id}$ on $U$ for some $p\geq 1$, then $f^{\circ p} = {\rm Id}$ on $\widehat{\mathbb{C}}$ by analytic continuation, which is not possible since $f^{\circ p}$ is a rational map of degree $d^p>1$); 
\item or $U$ is simply connected, in which case it is a Siegel disk; 
\item or $U$ is doubly connected, in which case it is a Herman ring. 
\end{itemize}

\noindent{\bf Case 1.2.} There is a sequence $(n_k)$ such that $f^{\circ n_k}\to \phi$ with $\phi$ non constant. By continuity, $\phi$ takes its values in $\overline U$. According to the Hurwitz Theorem, $\phi$ takes its values in $U$. Extracting a subsequence, we may assume that $n_{k+1}-n_k\to +\infty$ and that $f^{\circ (n_{k+1}-n_k)}\to \psi$. Then, $\psi\circ \phi = \phi$, so that $\psi$ is the identity on the image of $\phi$, thus on the whole component $U$ by analytic continuation. In other words, we are in the previous situation : $U$ is a Siegel disk or a Herman ring.

\noindent{\bf Case 2.}  Every limit value of the sequence $(f^{\circ n})$ is a constant. Note that if the sequence  
$(f^{\circ n_k})$ converges to $\alpha$, then $\alpha$ is a fixed point of $f$. Indeed, since $f(z)\in U$ for all $z\in U$, we have that 
\[\alpha\underset{k\to +\infty}\longleftarrow f^{\circ n_k}\bigl(f(z)\bigr) = f\bigl(f^{\circ n_k}(z)\bigr)\underset{k\to +\infty}\longrightarrow f(\alpha).\]
Note also that any point $z\in U$ can be joined to its image $f(z)\in U$ by a path compactly contained in $U$. It follows that the set of limit values of the sequence $\bigl(f^{\circ n}(z)\bigr)$ is a continuum.  Since  $f$ has finitely many fixed points, the whole sequence $(f^{\circ n})$ converges to a fixed point $\alpha$ of $f$.

\noindent{\bf Case 2.1.} If $\alpha\in U$, the derivative of $f^{\circ n}$ at $\alpha$ tends to $0$ as $n\to \infty$. As a consequence, $\alpha$ is a (super)attracting fixed point and $U$ is a (super)attracting domain.

\noindent{\bf Case 2.2.} If $\alpha\in \partial U$, then $\alpha$ cannot be (super)attracting since it belongs to the Julia set $\mathcal{J}_f$. It cannot be repelling since the sequence $(f^{\circ n})$ converges to $\alpha$. So, it is indifferent. We claim that ${\rm D}_\alpha f = 1$ and so, $U$ is a parabolic domain.

Fix $z_0\in U$ and set $z_n=f^{\circ n}(z_0)$. Let $\widetilde{\rho}:[0,1]\to U$ be a continuous path with $\widetilde{\rho}(0)= z_0$ and $\widetilde{\rho}(1)= z_1$. Then we can extend $\widetilde{\rho}$ to a continuous path $\rho:[0,+\infty)\to U$ by setting $\rho (n+t) = f^{\circ n}\circ \widetilde{\rho}(t)$ for all $n\in\mathbb{N}$ and all $t\in[0,1)$. Then $\lim_{t\to+\infty}\rho(t) = \alpha$ and $\rho$ satisfies $\rho(t+1)=f\circ \rho(t)$ for all $t\in [0,+\infty)$ and we conclude using the following lemma.

\begin{lemma}[Snail lemma]
Let $V$ be a neighborhood of the origin in $\mathbb{C}$ and let ${f: V\to f(V)}$ be a holomorphic function such that $f(\alpha)=\alpha$ and $f'(\alpha)\ne0$. If there is a continuous path $\rho:[0,+\infty)\to V\setminus\{\alpha\}$ such that $\lim_{t\to+\infty}\rho(t) = \alpha$ and $f\circ \rho(t)= \rho(t+1)$
for all $t\in [0,+\infty)$ then either $|f'(\alpha)|<1$ or $f'(\alpha)=1$.
\end{lemma}

This is a second place where we do not provide the details. For a proof of this result see for example \cite[Lemma 16.2]{Milnor1}.
\end{proof}

The following theorem, due to Fatou (1905), is probably the result
which started the entire field of holomorphic dynamics.

\begin{theorem}\label{critsattracted}
A (super)attracting domain always contains at least one critical point.
\end{theorem}

\begin{proof}
If $U$ is a superattracting domain, then it contains a fixed critical point, and the result is obvious. 
If $U$ is an attracting domain, it contains an attracting fixed point $\alpha$ with multiplier $\lambda$. We can show that there is a map $\phi:(\mathbb{C},0)\to (\widehat{\mathbb{C}},\alpha)$ which conjugates the multiplication by $\lambda$ to $f$. If $U$ does not contain any critical point, then $\phi$ extends to an entire map $\phi:\mathbb{C}\to U$ via the relation $\phi(\lambda z ) = f\circ \phi(z)$. This contradicts the Liouville Theorem. 
\end{proof}

\begin{remark}
The same is true for parabolic domains, but the proof is more difficult and requires a detailed study of the local theory near a fixed point with multiplier 1. A proof can be found for example in \cite[Theorem 10.15]{Milnor1}.
\end{remark} 

Fatou's classification of invariant Fatou components clearly provides also a classification of {\em periodic} Fatou components since they are invariant for an iterate of the rational map, and even {\em pre-periodic} Fatou components, that is components $U$ so that there exist $m\ge0$ and $p\ge 1$ such that $f^{\circ (m+p)} (U) = f^{\circ m}(U)$. The following fundamental result due to Sullivan completes the description for rational maps of degree $d\ge 2$.

\begin{theorem}[Sullivan \cite{Su}]
Let $f:\widehat{\mathbb{C}}\to \widehat{\mathbb{C}}$ be a rational map of degree $d\ge 2$. Then every Fatou component of $f$ is pre-periodic.\label{thm:2}
\end{theorem}

The proof of Sullivan's non-wandering Theorem~\ref{thm:2} strongly relies on the Ahlfors-Bers mesurable mapping Theorem for quasi-conformal functions and we refer to the original paper of Sullivan \cite{Su} for it. The recent notes \cite{XavierKAWA} present a proof due to Adam Epstein and based on a density result of Bers for quadratic differentials. Such results are strongly one-dimensional and do not have an analogue in higher dimension, making impossible to mimic Sullivan's proof there. Besides this observation, little was known about this problem until recently. In Section \ref{sec:8} we will present some recent results on Fatou components in dimension 2.

\section{Parabolic implosion in dimension 1}
\label{sec:7}

In this section, we recall the main ingredients in the proof that the Julia set ${\cal J}_f$ does not depend continuously on $f$ for the Hausdorff topology on the space of compact subsets of $\mathbb{C}$. Lavaurs proved the following result (see also \cite{douady}). 
\begin{theorem}[Lavaurs \cite{Lavaurs}]
Assume $f:\mathbb{C}\to \mathbb{C}$ is a polynomial fixing $0$ with $f(z) = z+ z^2 + {O}(z^3)$. Then, 
\[{\cal J}_f\subsetneq \limsup_{\delta\to 0} {\cal J}_{f+\delta}.\]
\end{theorem}

The proof is based on the description of limits of iterates of $f+\delta$ in terms of maps called {\em Fatou coordinates} that will be introduced in subsection \ref{subsec:7.2}. Those limits are called {\em Lavaurs maps}. 

\subsection{Parabolic basin}
\label{subsec:7.1}

In the rest of this section, we assume that $f:\mathbb{C}\to \mathbb{C}$ is a polynomial whose expansion near $0$ is of the form 
\[f(z) = z+ z^2 + az^3 + {O}(z^4)\quad\text{with}\quad a\in \mathbb{C}.\]
To understand the local dynamics of $f$ near $0$, it is convenient to consider the change of coordinates $Z:=-1/z$. In the $Z$-coordinate, the expression of $f$ becomes 
\[F(Z)= Z + 1 + \frac{b}{Z} + {O}\left(\frac{1}{Z^2}\right)\quad\text{with}\quad b:=1-a.\]
In particular, if $R>0$ is large enough, $F$ maps the right half-plane $\{\Re(Z)>R\}$ into itself, so that if $r>0$ is close enough to $0$, $f$ maps the disk $D(-r,r)$ into itself. In addition, the orbit under $f$ of any point $z\in D(-r,r)$ converges to $0$ tangentially to the real axis. 
Similarly, if $r>0$ is sufficiently close to $0$, there is a branch of $f^{-1}$ which maps the disk $D(r,r)$ into itself and the orbit under that branch of $f^{-1}$ of any point $z\in D(r,r)$ converges to $0$ tangentially to the real axis. 

\begin{definition}
The {\em basin} ${\cal B}_f$ is the open set of points whose orbit under iteration of $f$ intersects the disks $D(-r,r)$ for all $r>0$. 
\end{definition}


\subsection{Fatou coordinates}
\label{subsec:7.2}

In order to understand further the local dynamics of $f$ near $0$, it is customary to use local {\em attracting and repelling Fatou coordinates}. In the case of a polynomial, those Fatou coordinates have global properties.

\begin{proposition}
There exists a (unique) {\em attracting Fatou coordinate} $\Phi_f:{\cal B}_f\to \mathbb{C}$ which semi-conjugates $f:{\cal B}_f\to {\cal B}_f$ to the translation $T_1:\mathbb{C}\ni Z\mapsto Z+1\in \mathbb{C}$: 
\[\Phi_f\circ f= T_1\circ \Phi_f.\] 
and satisfies the normalization: 
\[\Phi_f(z) = -\frac{1}{z} - b\log \left(-\frac{1}{z}\right)+ o(1)\quad \text{as}\quad \Re\left(-\frac{1}{z}\right)\to+\infty.\]
\end{proposition}

\begin{proof}
For $z\in {\cal B}_f$, set 
\[Z:= -\frac{1}{z},\quad Z_n:= -\frac{1}{f^{\circ n}(z)}\quad \text{and}\quad 
\varPhi_n(z) :=  Z_n- n - b\log Z_n.\]
We have that
\[\Re(Z_n)\to +\infty\quad \text{and}\quad Z_{n+1} = Z_n + 1 +\frac{b}{Z_n} +  {O}\left(\frac{1}{Z_n^2}\right),\]
so that 
\[\frac{1}{Z_n} = {O}\left(\frac{1}{Z+n}\right)\quad \text{as}\quad n\to +\infty.\]
As a consequence, 
\[\varPhi_{n+1}(z) - \varPhi_n(z) = Z_{n+1} - Z_n-1 - b\log\frac{Z_{n+1}}{Z_n} = {O}\left(\frac{1}{Z_n^2}\right) 
= {O}\left(\frac{1}{(Z+n)^2}\right).\]
So, the sequence $(\varPhi_n)$ converges to a limit $\Phi_f:{\cal B}_f\to \mathbb{C}$ which satisfies 
\[\Phi_f(z) = \varPhi_0(z) + {O}\left(\frac{1}{Z}\right) = -\frac{1}{z} - b\log \left(-\frac{1}{z}\right) + o(1)\quad \text{as}\quad \Re\left(-\frac{1}{z}\right)\to+\infty.\]
Passing to the limit on the relation $\varPhi_n\circ f = T_1\circ \varPhi_{n+1}$
yields the required result.  
\end{proof}

Figure \ref{fig:fatou} gives a rough idea of the behaviour of the Fatou coordinate $\Phi_f$ for the cubic polynomial $f(z):=z+z^2+0.95z^3$. The basin ${\cal B}_f$ contains the two critical points of $f$. Those points and their iterated preimages form the critical points of $\Phi_f$. Denote by $c^+$ the critical point with positive imaginary part and by $c^-$ its complex conjugate. Set $v^\pm:=\Phi_f(c^\pm)$. Points $z\in {\cal B}_f$ are colored according to the location $\Phi_f(z)$: dark grey when $\Im\bigl(\Phi_f(z)\bigr)<\Im(v^-)$, light grey when  $\Im(v^-)<\Im\bigl(\Phi_f(z)\bigr)<\Im(v^+)$ and medium grey when $\Im(v^+)<\Im\bigl(\Phi_f(z)\bigr)$.


\begin{proposition}
There exists a (unique) {\em repelling Fatou parametrization} $\Psi_f:\mathbb{C}\to \mathbb{C}$ which semi-conjugates $T_1:\mathbb{C}\to \mathbb{C}$ to $f:\mathbb{C}\to \mathbb{C}$: 
\[\Psi_f\circ T_1= f\circ \Psi_f.\] 
and satisfies the normalization:
\[\Psi_f(Z)=-\frac{1}{Z+ b\log(-Z)+o(1)}\quad\text{as}\quad\Re(Z)\to-\infty.\]
\end{proposition}

\begin{proof}
Choose $r>0$ sufficiently close to $0$ so that there is a branch $g$ of $f^{-1}$ which maps the disk $D(r,r)$ into itself. For $z\in D(r,r)$, set 
\[Z:= -\frac{1}{z},\quad Z_n := -\frac{1}{g^{\circ n}(z)}\quad\text{and}\quad \varPhi_n(z):= Z_n+n -b\log(-Z_n).\]
Note that
\[Z_n = Z_{n+1} + 1 + \frac{b}{Z_{n+1}} + {O}\left(\frac{1}{Z_{n+1}^2}\right). \]
So, as in the previous proof
\[\varPhi_{n+1}(z)-\varPhi_n(z) = Z_{n+1} - Z_n+1 + b\log\frac{Z_{n+1}}{Z_n} = {O}\left(\frac{1}{Z_{n+1}^2}\right) = {O}\left(\frac{1}{(Z-n)^2}\right).\]
The sequence $\varPhi_n$ converges in the left half-plane $\bigl\{\Re(Z)<-1/(2r)\bigr\}$ to a limit $\Phi_g$ which satisfies
\[\Phi_g(z) = \varPhi_0(z) + {O}\left(\frac{1}{Z}\right) = Z - b\log (-Z)+ {o}(1)\quad \text{as}\quad \Re\left(-\frac{1}{z}\right)\to-\infty.\]
Passing to the limit on the equation $\varPhi_{n+1}\circ f = T_1\circ \Phi_n$
shows that $\Phi_g$ conjugates $f$ to $T_1$. 
The inverse $\Psi_f$ of $\Phi_g$ conjugates $T_1$ to $f$ and 
\[Z = \Phi_g\circ \Psi_f(Z) = -\frac{1}{\Psi_f(Z)} -b\log \left(-\frac{1}{\Psi_f(Z)}\right)+ { o}(1) = -\frac{1}{\Psi_f(Z)} -b\log (-Z)+ { o}(1)\]
as $\Re(Z)\to -\infty$. 
\end{proof}

\subsection{Lavaurs maps}
\label{subsec:7.3}

\begin{definition}
The {\em Lavaurs map} ${\cal L}_f:{\cal B}_f\to \mathbb{C}$ is the map 
\[{\cal L}_f:=\Psi_f\circ \Phi_f:{\cal B}_f\to \mathbb{C}.\]
\end{definition}

Note that the Lavaurs map ${\cal L}_f$ commutes with $f$. Indeed, 
$\Phi_f\circ f = T_1\circ \Phi_f$ and $\Psi_f\circ T_1 = f\circ \Psi_f$, so that
\[{\cal L}_f\circ f = \Psi_f\circ  \Phi_f\circ f = \Psi_f\circ T_1\circ \Phi_f =  f\circ \Psi_f\circ \Phi_f= f\circ {\cal L}_f.\]

Figure \ref{fig:lavaurs} gives a rough idea of the behaviour of the Lavaurs map ${\cal L}_f$ for the cubic polynomial $f(z):=z+z^2+0.95z^3$. Points in the basin ${\cal B}_f$ are colored according to the location of their image by ${\cal L}_f$: grey when ${\cal L}_f(z)\in {\cal B}_f$, white otherwise. The restriction of ${\cal L}_f$ to each bounded white domain is a covering of
$\mathbb{C}\setminus \overline{\cal B}_f$.

\begin{figure}
\centerline{
\includegraphics[height=7cm]{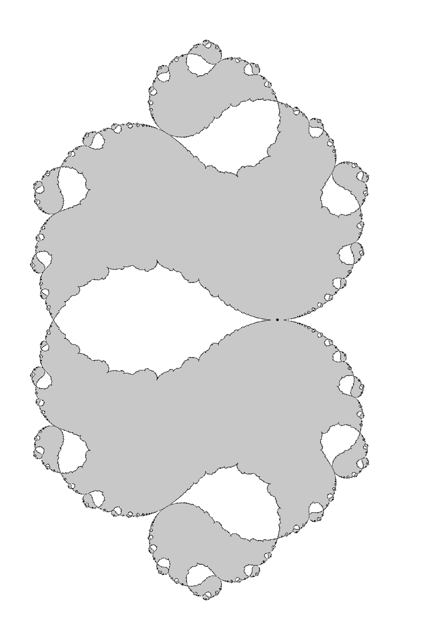}
}
\caption{Behavior of  ${\cal L}_f$ for  $f(z):=z+z^2+0.95z^3$. } \label{fig:lavaurs}
\end{figure}

\subsection{Discontinuity of the Julia set}
\label{subsec:7.4}

Set 
\[{\cal K}({\cal L}_f):=\bigcap_{n\geq 0} {\cal L}_f^{-n}({\cal K}_f)\quad\text{and}\quad {\cal J}({\cal L}_f):=\partial {\cal K}({\cal L}_f).\]
Figure \ref{fig:enriched} shows ${\cal K}({\cal L}_f)$ for the cubic polynomial $f(z) = z+z^2+0.95 z^3$. The Lavaurs map ${\cal L}_f$ has  two complex conjugate sets of attracting fixed points. The fixed points of ${\cal L}_f$ are indicated and their basins of attraction are colored (dark grey for one of the fixed points, and light grey for the others). Those basins form the interior of ${\cal K}({\cal L}_f)$. The black set ${\cal J}({\cal L}_f)$ is the topological boundary of ${\cal K}({\cal L}_f)$.

\begin{figure}[htbp]
\centerline{
\includegraphics[height=7cm]{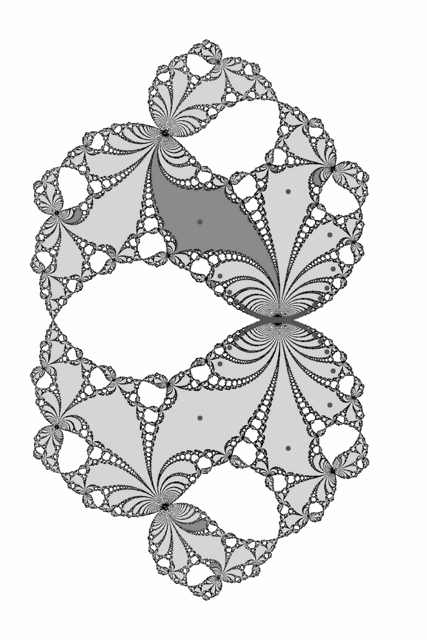}
}
\caption{The set ${\cal K}({\cal L}_f)$ for $f(z) = z+ z^2 + 0.95z^3$.}\label{fig:enriched}
\end{figure}

The following result may be considered as the main reason why Lavaurs maps are studied in holomorphic dynamics.

\begin{proposition}[Lavaurs \cite{Lavaurs}]
Let $f:\mathbb{C}\to \mathbb{C}$ be a polynomial whose expansion at~$0$ is $f(z) = z + z^2 + {O}(z^3)$. Let $(N_n)$ be a sequence of integers tending to $+\infty$ and $(\varepsilon_n)$ be a sequence of complex numbers tending to $0$, such that 
\[N_n-\frac{\pi}{\varepsilon_n}\to  0.\]
Then, 
\[(f+\varepsilon_n^2)^{\circ N_n} \to {\cal L}_f\quad\text{locally uniformly on }{\cal B}_f.\]
In addition, 
\[{\cal J}_f\subsetneq {\cal J}({\cal L}_f)\subseteq \liminf {\cal J}_{f+\varepsilon_n^2} \quad\text{and}\quad 
\limsup {\cal K}_{f+\varepsilon_n^2}\subseteq {\cal K}({\cal L}_f)\subsetneq {\cal K}_f.\]\label{prop:lavaurs}
\end{proposition}

We do not present the proof of this result which is rather technical, and can be found in \cite{douady} or \cite{shishikura} for example. Instead, we will show that the result holds for the M\"obius transformation 
\[g(z) = \frac{z}{1-z} = z+z^2 + {O}(z^3).\]
In that case, all maps involved are M\"obius transformations and the computations are explicit. 
If we perform the change of coordinates $Z=-1/z$, the M\"obius transformation $g$ gets conjugated to $T_1$. Thus,
\[\Phi_g(z) = -\frac{1}{z}\quad \text{and}\quad \Psi_g(Z) = -\frac{1}{Z},\quad \text{so that}\quad{\cal L}_g = {\rm Id}.\]
Note that $g+\varepsilon^2$ is also a M\"obius transformation. It has two fixed points 
\[\alpha^\pm= \pm i\varepsilon + {O}(\varepsilon^2)\quad \text{with multipliers}\quad  \lambda^\pm = \exp\bigl(\pm2 i\varepsilon + {O}(\varepsilon^3)\bigr).\]
So, if $N\to +\infty$ and 
\[N-\frac{\pi}{\varepsilon}\to 0,\quad\text{so that}\quad \varepsilon = \frac{\pi}{N + { o}(1)} = \frac{\pi}{N} +  { o}\left(\frac{1}{N^2}\right),\]
then $(g+\varepsilon^2)^{\circ N}$ is a M\"obius transformation fixing $\alpha^\pm$ with multipliers 
\[\mu^\pm:=(\lambda^\pm)^N = \exp\left(\pm2\pi i  + { o}(1/N)\right) = 1  + { o}\left(\frac{1}{N}\right).\]
This M\"obius transformation is 
\[z\mapsto \alpha^+ + \frac{\mu^+ \cdot (z-\alpha^+)}{\displaystyle 1-\frac{\mu^+-1}{\alpha^+-\alpha^-}(z-\alpha^+)}.\]
Since 
\[\mu^+-1 = { o}\left(\frac{1}{N}\right) = { o} (\alpha^+-\alpha^-),\]
we see that indeed, 
\[(g+\varepsilon^2)^{\circ N}\underset{N\to +\infty}\longrightarrow  {\cal L}_g.\]

Recently, in \cite{Vi2}, Vivas studied a non-autonomous parabolic implosion in dimension 1 for the map $f(z)=\frac{z}{1-z}$.

In higher dimension, (semi-)parabolic implosion was recently studied for dissipative polynomial automorphisms of $\mathbb{C}^2$ by Bedford, Smillie and Ueda in \cite{bsu} (see also \cite{dl}) and their strategy was recently adapted by Bianchi in \cite{bianchi} for a perturbation of a class of holomorphic endomorphisms tangent to identity, establishing a two-dimensional Lavaurs theorem for such a class.

\section{Fatou components for polynomial maps in dimension 2}
\label{sec:8}

We end these notes by giving an updated account of the recent results on Fatou components for polynomial skew-products in complex dimension two in a neighbourhood of a periodic fiber. We divide our discussion according to the different possible kinds of periodic fibers.

\subsection{Preliminaries}
\label{subsec:8.1}

Let $F:\mathbb{C}^2\to \mathbb{C}^2$ be a holomorphic endomorphism of $\mathbb{C}^2$, and consider the discrete holomorphic dynamical system given by the iteration of $F$. In the investigation of the global behaviour of such a system it is natural to give the generalize the definition of the {\em Fatou set of $F$} as the largest open set $\mathcal{F}(F)$ where the family of iterates $(F^{\circ n})_{n\in\mathbb{N}}$ of $F$ is  
normal. A connected component of the Fatou set is called a {\em Fatou component}.

We have seen in the previous sections that in complex dimension one, Fatou components of rational maps of degree at least $2$ on the Riemann sphere are well understood, thanks to Theorem \ref{classificationFatou} and Theorem \ref{thm:2}.

In complex dimension two, the understanding of Fatou components is far less complete. A considerable progress in the classification of periodic Fatou components has been achieved thanks to Bedford and Smillie \cite{BS1} \cite{BS2} \cite{BS3}, Forn\ae ss and Sibony \cite{FS}, Lyubich and Peters \cite{LP} and Ueda \cite{Ue}. 

The question of the existence of wandering (i.e., not pre-periodic) Fatou components in higher dimension was put forward by several authors since the 1990's (see e.g., \cite{FS2}). Higher dimensional \emph{transcendental} (i.e., non polynomial) maps with wandering domains can be constructed from one-dimensional examples by taking direct products.  An example of a  transcendental {\em biholomorphism} of $\mathbb{C}^2$ with  a wandering Fatou component oscillating to infinity was constructed by Forn\ae ss and   Sibony  in \cite{FS1}. Nonetheless, until recently very little was known about the existence of wandering Fatou components for holomorphic endomorphisms of $\mathbb{P}^2(\mathbb{C})$ or for polynomial endomorphisms of $\mathbb{C}^2$.

A first natural class of polynomial endomorphisms of $\mathbb{C}^2$ to consider are {\em direct product}, that is maps $F:  \mathbb{C}^2\to \mathbb{C}^2$ of the form
$$
F(z,w) = ( f(z), g(w)),
$$ 
where $f$ and $g$ are complex polynomials in one variable. This allows us to recover the generalizations of one-dimensional dynamical behaviours in dimension two. However, direct products are a very particular class and they do not give us a complete understanding of all possible behaviours of polynomial endomorphisms in~$\mathbb{C}^2$. 

A more interesting class to consider is given by polynomial {\em skew-products} in $\mathbb{C}^2$, namely polynomial maps $F: \mathbb{C}^2\to\mathbb{C}^2$ of the form
\begin{equation}\label{eq:1}
F(z,w) 
=
(f(z,w), g(w)),
\end{equation}
where $g$ is a complex polynomial in one variable and $f$ is a complex polynomial in two variables.
Since they leave invariant the fibration $\{w =\hbox{const.}\}$, skew-products allow us to {\em build on one-dimensional dynamics} and to get a first flavour of the richness of the higher dimension setting we are working in. This idea has been used by several authors to construct maps with particular dynamical properties. Dujardin, for example, used in \cite{D} specific skew-products to construct a non-laminar Green current. Boc-Thaler, Forn\ae ss and Peters constructed in \cite{BFP} a map having a Fatou component with a punctured limit set. Last but not least, as we will explain in subsection \ref{subsec:8.2}, skew-products are one of the key ingredients in the construction we recently obtained  in \cite{ABDPR} with Astorg, Dujardin, and Peters of holomorphic endomorphisms of $\mathbb{P}^2(\mathbb{C})$ having a wandering Fatou component.

The investigation of the holomorphic dynamics of polynomial skew-products was started by Heinemann \cite{He} and then continued by Jonsson \cite{J}. The topology of Fatou components of skew-products has been studied by Roeder in \cite{Ro}.

\smallskip
Given a Fatou component $\Omega$ of a polynomial skew-product $F$ in $\mathbb{C}^2$, its projection on the second coordinate $\Omega_2=\pi_2(\Omega)$ is a Fatou components for $g$ and hence thanks to Sullivan's non-wandering Theorem~\ref{thm:2}, up to considering an iterate of $F$, it has to fall into one of the three cases given by Theorem~\ref{classificationFatou}, and moreover, since we are considering polynomials, Herman rings cannot occur.
Therefore, since (pre-)periodic points for $g$ correspond to (pre-)periodic fibers for $F$, up to considering an iterate of $F$, we can restrict ourselves to study what happens in neighbourhoods of invariant fibers of the form $\{w=w_0\}$. 
One-dimensional theory also describes the dynamics on the invariant fiber, which is given by the action of the one-dimensional polynomial $f(z,w_0) := f_{w_0}(z)$, and hence the Fatou components of $f_c$ will be again all pre-periodic and, up to consider an iterate, we can assume that they are either attracting basins, or parabolic basins or Siegel disks. 
This structure leads us to two immediate questions.

\begin{enumerate}
\item  Do all Fatou components of $f_c$ bulge to two-dimensional Fatou components of $F$?
\item Is it possible to have wandering Fatou components for $F$ in a neighbourhood of an invariant fiber?
\end{enumerate}

In the following we shall call an invariant fiber $\{w=w_0\}$ {\em attracting, parabolic} or {\em elliptic} according to whether $w_0$ is an attracting, parabolic or elliptic fixed point for $g$. A {\em bulging} Fatou component will be a Fatou component $\Omega$ of $F$ such that $\Omega\cap\{w=w_0\}$ is a one-dimensional Fatou component of $f_{w_0}$ on the invariant fiber $\{w=w_0\}$. We shall say that a Fatou component $\Omega_{w_0}$ of $f_{w_0}$ on the invariant fiber $\{w=w_0\}$ {\em is bulging} if there exists a bulging Fatou component $\Omega$ of $F$ so that $\Omega_{w_0}=\Omega\cap\{w=w_0\}$.

\subsection{Attracting invariant fiber}
\label{subsec:8.2}

Let us consider a polynomial skew-product $F: \mathbb{C}^2\to \mathbb{C}^2$ of degree $d\ge 2$
$$
F(z,w) 
=
(f(z,w), g(w)),
$$
with an {\em attracting invariant fiber}. We can assume without loss of generality that the invariant fiber is $\{w=0\}$. Therefore we have $g(0)=0$ and $|g'(0)|<1$. In this case we know 
that there exists an {\em attracting basin}, containing the origin, of points whose iterates converge to the origin. The rate of convergence to the fixed point depends on whether the fixed point is {\em superattracting}, i.e, $g'(0) = 0$, or {\em attracting} or {\em geometrically attracting}, i.e., $g'(0) \ne 0$. 

\subsubsection{Superattracting case}

This setting was studied by Lilov in \cite{Li} who was able to answer both questions stated in the introduction. He first proved the following result giving a positive answer to Question 1.

\begin{theorem}[Lilov \cite{Li}]
Let $F: \mathbb{C}^2\to\mathbb{C}^2$ be a polynomial skew-product of the form \eqref{eq:1} of degree $d\ge 2$. Let $\{w=w_0\}$ be a superattracting invariant fiber for $F$. Then all one-dimensional Fatou components of $f_{w_0}$ bulge to Fatou components of $F$. 
\label{thm:3bis}
\end{theorem}
  
\begin{proof}[Idea of the proof]
We can assume $w_0=0$ without loss of generality. Thanks to Theorem~\ref{classificationFatou} and Theorem~\ref{thm:2}, all Fatou components of the restriction $f_0(z) = f(z,0)$ of $f(z,w)$ to the invariant fiber are  (pre-)periodic and are either attracting basins, or parabolic basins or Siegel domains. The strategy of the proof is to  prove separately for each of these cases that the corresponding component is contained in a two-dimensional Fatou component of $F$. 
The bulging of one-dimensional Fatou components of attracting periodic points of $f_0(z)$ is well-known and follows for instance from the results of Rosay and Rudin \cite{RR}.
For the remaining cases, by \cite[Theorem 3.17]{Li} there exists a strong stable manifold through all point in the one-dimensional Fatou components of parabolic or elliptic periodic points of $f_0(z)$, and so the corresponding bulging Fatou components simply consist of the union of such manifolds. 
\end{proof}

Then Lilov proved the following result implying the non-existence of wandering Fatou components in a neighbourhood of a superattracting invariant fiber.

\begin{theorem}[Lilov \cite{Li}]
Let $F: \mathbb{C}^2\to\mathbb{C}^2$ be a polynomial skew-product of the form \eqref{eq:1} of degree $d\ge 2$. Let $\{w=w_0\}$ be a superattracting invariant fiber for $F$ and let $\mathcal{B}$ be the immediate basin of the superattracting fixed point $w_0$. Take $w\in\mathcal{B}$ and let $D_{w}$ be a one-dimensional open disk lying in the fiber over $w$. Then the forward orbit of $D_w$ must intersect one of the bulging Fatou components of $f_{w_0}$.\label{thm:3}
\end{theorem}
    
The proof relies on the repeated use of \cite[Lemma 3.2.4]{Li} applied to the orbit of a disk lying in a fiber over a point in the attracting basin, in order to obtain estimates from below for the radii of the images. 
Thanks to \cite[Proposition 3.2.8]{Li}, by studying the geometry of the bulging Fatou components, it is also possible to obtain an upper bound on the largest possible disk lying in a fiber over a point in the attracting basin that can lie in the complement of a bulging Fatou component, depending on the distance to the invariant fiber.
The conclusion then follows combining these two estimates.

All bulging Fatou components are (pre-)periodic, therefore all Fatou components for $F$ in a neighbourhood of a superattracting invariant fiber are (pre-)periodic, and then the non-existence of wandering Fatou components in a neighbourhood of a superattracting invariant fiber follows immediately.

\begin{corollary}[Lilov \cite{Li}]
Let $F: \mathbb{C}^2\to\mathbb{C}^2$ be a polynomial skew-product of the form \eqref{eq:1} of degree $d\ge 2$. Let $\{w=w_0\}$ be a superattracting invariant fiber for $F$ and let $\mathcal{B}$ be the immediate basin of the superattracting fixed point $w_0$. Then there are no wandering Fatou components in $\mathcal{B}\times\mathbb{C}$.
\end{corollary}

\subsubsection{Geometrically attracting case}

The geometrically attracting case was first partially addressed by Lilov in \cite{Li} even if not stated explicitly. In fact, the proof of Theorem~\ref{thm:3bis} can be readily adjusted to this case obtaining the following statement answering Question 1.

\begin{theorem}[Lilov \cite{Li}]
Let $F: \mathbb{C}^2\to\mathbb{C}^2$ be a polynomial skew-product of the form \eqref{eq:1} of degree $d\ge 2$. Let $\{w=w_0\}$ be an attracting invariant fiber for $F$. Then all one-dimensional Fatou components of $f_{w_0}$ bulge to Fatou components of $F$. \label{thm:4bis}
\end{theorem}

On the other hand, the proof of Theorem~\ref{thm:3} cannot be generalized to this setting, which is indeed more complicated than the superattracting case. 
In fact, Theorem~\ref{thm:3} does not hold in general, as showed by Peters and Vivas with the following result.

\begin{theorem}[Peters-Vivas \cite{PV}]
Let $F: \mathbb{C}^2\to\mathbb{C}^2$ be a polynomial skew-product of the form
\begin{equation}
\label{eq:2}
F(z,w) = (p(z) + q(w) , \lambda w),
\end{equation}
with $0<|\lambda|<1$ and $p$ and $q$ complex polynomials. Then there exists a triple $(\lambda, p, q)$ and a holomorphic disk $D\subset\{w=w_0\}$ whose forward orbit accumulates at a point $(z_0, 0)$, where $z_0$ is a repelling fixed point in the Julia set of $f_0$. 
\label{thm:5}
\end{theorem}

As a consequence, the forward orbits of $D$ cannot intersect the bulging Fatou components of $f_0$. 

The family $(F|_D^{\circ n})_{n\in\mathbb{N}}$ is normal on the disks $D$, and so these are Fatou disks. However such disks are completely contained in the Julia set of $F$, which is the complement in $\mathbb{C}^2$ of the Fatou set (see \cite[Theorem 6.1]{PV}).

The geometrically attracting case has been further investigated by Peters and Smit in \cite{PS}. 
They focused their investigation on polynomial skew-products such that the action on the invariant attracting fiber is {\em subhyperbolic}, that is the polynomial does not have parabolic periodic points and all critical points lying on the Julia set are pre-periodic. They proved the following result.

\begin{proposition}[Peters-Smit \cite{PS}]
Let $F: \mathbb{C}^2\to\mathbb{C}^2$ be a polynomial skew-product of the form \eqref{eq:1}. Assume that the origin is an attracting, not superattracting, fixed point for $g$ with corresponding basin $B_g$, and the polynomial $f_0(z) := f(z,0)$ is subhyperbolic. Then there exists a set $E\subset\mathbb{C}$ of full mesure, such that for every $\widetilde w_0\in E$ the forward orbit of every disk in the fiber $\{w=\widetilde w_0\}$ must intersect a bulging Fatou component of $f_0$.
\label{prop:1}
\end{proposition}

\begin{proof}[Idea of the proof]
Notice that it suffices to prove the proposition in a neighbourhood of the attracting fiber $\{w=0\}$. Therefore, up to considering a smaller neighbourhood, we can assume without loss of generality that $g(w) = \lambda w$, and
$$
f(z,w) = a_0(w) + a_1(w) z + \cdots + a_d(w) z^d
$$
where $a_0(w), \dots, a_d(w)$ are holomorphic functions in $w$. The subhyperbolicity of the polynomial $f_0$ implies that its Fatou set is the union of finitely many attracting basins, and the orbits of the critical points contained in the Fatou set converge to one of these attracting cycles. 
The proof can be divided into 5 main steps.

\smallskip\noindent{\em Step 1.} Fix $R>0$ large enough so that for all $z$ such that $|z|>R$ we have $|f_0(z)|>2|z|$ and set 
$$
W_0 = \{|z|>R\} \cup \bigcup_{y\in {\rm Att}(f_0)} W_y
$$
where ${\rm Att}(f_0)$ is the set of all attracting periodic points of $f_0$, and for each $y\in {\rm Att}(f_0)$ the set $W_y$ is an open neighbourhood of the orbit of $y$ such that $\overline{f_0(W_y)}\subset W_y$. 
Fix a neighbourhood $U$ of the post-critical set of $f_0$. Then by \cite[Proposition 15]{PS}, there exists a set $E\subset \mathbb{C}$ of full mesure in a neighbourhood of the origin such that for all $\widetilde w_0\in E$ there exists a constant $C = C(\widetilde w_0,U)$ such that for all $n\in\mathbb{N}$ we have
\begin{equation}
\label{eq:3}
{\rm Card}\left\{z:  \frac{\partial F_1^{\circ n}}{\partial z}(z,\widetilde w_0) = 0~\hbox{and}~F_1^{\circ n}(z,\widetilde w_0)\not\in W_0\times U\right\} \le C \sqrt{n}, 
\end{equation}
where $F_1^{\circ n}$ is the first component of the $n$-th iterate of $F$.

\smallskip\noindent{\em Step 2.} Assume by contradiction that a fiber $\{w = \widetilde w_0\}$, with $\widetilde w_0\in E$,  contains a disk $D$ whose forward orbit avoids the bulging Fatou components of $f_0$. Then the restriction of $F^{\circ n}$ to $D$ is bounded and hence a normal family. Therefore, up to shrinking $D$ there exists a subsequence $F^{\circ n_j}$ such that $F^{\circ n_j}|_D$ converges, uniformly on compact subsets of $D$, to a point $\zeta$ in the Julia set of $f_0$. Moreover, there exists $\varepsilon >0$ so that $F^{\circ n}(D)\cap(W_0\times D(0,\varepsilon))$ is empty for all $n\in\mathbb{N}$. 

\smallskip\noindent{\em Step 3.} Each critical point $x$ contained in the Julia set is eventually mapped into a repelling periodic point, and up to considering an iterate of $F$ we may assume that it is  eventually mapped into a repelling fixed point with multiplier $\mu$, with $|\mu|>1$. 
The main tool to control the orbits of the critical points of $F$ is obtained using a linearization map of the unstable manifold of the repelling fixed point, given by a map $\Phi: \mathbb{C}\to\mathbb{C}$ satisfying $\Phi(\mu t) = f_0^{\circ k}\circ \Phi(t)$ for some $k\in \mathbb{N}$. Thanks to \cite[Proposition 10]{PS}, there exist $\widetilde C>1$ and $0<\gamma<1$ so that
\begin{equation}
\label{eq:4}
{\rm Area}(F^{\circ n}(D))\le \widetilde C\gamma^n.
\end{equation}

\smallskip\noindent{\em Step 4.} We may assume that $\zeta$ does not lie in the post-critical set, and we may choose $U$ and $r>0$ such that $D(\zeta, r)\cap (U\cup W_0)=\emptyset$.
Let $j_1\in\mathbb{N}$ be such that $F_1^{\circ n_j}(D) \subseteq D(\zeta, \frac{r}{2})$ for all $j\ge j_1$, and consider $O_j$ the connected component of $(F^{\circ n_j})^{-1}(D(\zeta, r)\times\{\lambda^{n_j}\widetilde w_0\})$ containing $D$. Then $D \subseteq O_j \subseteq D(0,R) \times\{\widetilde w_0\}$, and we can study the proper holomorphic function $F_1^{\circ n_j}:  O_j\to D(\zeta, r)$. Thanks to \eqref{eq:3}, such a map has at most $d_j = C \sqrt{n_j}$ critical points. 

\smallskip\noindent{\em Step 5.} It is possible (see \cite[Proposition 28]{PS}) to find a uniform constant $C_1>0$ so that if $f: \mathbb{D}\to\mathbb{D}$ is a proper holomorphic function of degree $d$, the set $R\subset \mathbb{D}$ has Poincar\'e area equal to $A$, and $d\cdot A^{1/2d}<8$, then the Poincar\'e area of $f^{-1}(R)$ is at most $C_1 d^3 A^{1/d}$. Then, setting $R_j = F_1^{\circ n_j}(D)$ and denoting by $A_j$ its Poincar\'e area ${\rm Area}_{D(\zeta,r)}(R_j)$ with respect to $D(\zeta,r)$, for $j \ge j_1$, we have $R_j\subseteq D(\zeta,r)$, and we can estimate $A_j$ applying \eqref{eq:4}. Therefore there exists $j_2\ge j_1$ such that $d_j A_j^{1/2d_j}<1/8$ for all $j\ge j_2$. This implies
$$
{\rm Area}_{D(0,R)}(D)\le {\rm Area}_{O_j}(D)\le C_2 d_j^3 A_j^{{1/d_j}} \le M n_j^{3/2} \gamma^{n_j^{3/2}}
$$
where $M>0$. The contradiction follows from the fact that the last expression will converge to zero as $j$ increases towards infinity.
\end{proof}

Thanks to the fact that in particular  $E$ is dense, Peters and Smit are able to give a negative answer to Question 2 when the action on the invariant fiber is subhyperbolic. They also obtain as a corollary that the only Fatou components of $F$ are the bulging ones, since the topological degree of $F$ equals the one of $f_0$, implying that the only Fatou components that can be mapped onto the bulging Fatou components of $f_0$ are exactly those bulging Fatou components.

\begin{theorem}[Peters-Smit \cite{PS}] 
Let $F: \mathbb{C}^2\to\mathbb{C}^2$ be a polynomial skew-product of the form \eqref{eq:1}. Assume that the origin is an attracting fixed point for $g$ with corresponding basin ${\cal B}_g$, and the polynomial $f_0(z) := f(z,0)$ is subhyperbolic. Then $F$ has no wandering Fatou component over ${\cal B}_g$.
\end{theorem}

Very recently, Ji was able in \cite{Ji} to generalize Lilov's Theorem \ref{thm:3} to polynomial skew-products with an invariant geometrically attracting fiber under the hypothesis that the multiplier of the invariant fiber is sufficiently small. More precisely he proved the following result.

\begin{theorem}[Ji \cite{Ji}]
Let $F: \mathbb{C}^2\to\mathbb{C}^2$ be a polynomial skew-product of the form \eqref{eq:1} of degree $d\ge 2$. Let $\{w=w_0\}$ be an attracting invariant fiber for $F$ and let $\mathcal{B}_{w_0}$ be the  basin of the attracting fixed point $c$. Then there exists $\lambda_0(w_0, f)>0$ depending only on $f$ and $c$ such that if $|g'(w_0)|<\lambda_0$, then there are no wandering Fatou components in $\mathcal{B}_{w_0}\times \mathbb{C}$.
\label{thm:Ji}
\end{theorem}

The proof of this result follows Lilov's strategy. The main difficulty is due to the breaking down of Lilov's argument in the geometrically attracting case as we pointed out at the beginning of this section. Ji is able to overcome such difficulty by adapting a one-dimensional result due to Denker, Przytycki and Urbanski in \cite{DPU} in this case. Such result is used to obtain estimates of the size of bulging Fatou components and of the size of forward images of wandering Fatou disks.

\smallskip
Ji also proved in \cite{Ji2} that if $F: \mathbb{C}^2\to\mathbb{C}^2$ is a polynomial skew-product with an attracting invariant fiber $L$ such that the restriction of $F$ to $L$ is non-uniformly hyperbolic, then the Fatou set in the basin of $L$ coincides with the union of the basins of attracting cycles, and the Julia set in the basin of $L$ has Lebesgue measure zero. Therefore there are no wandering Fatou components in the basin of the invariant attracting fiber.
 
\subsection{Parabolic invariant fiber and wandering domains}
\label{subsec:8.3}

A first contribution to the investigation of this case is due to Vivas, who proved a parametrization result  \cite[Theorem 3.1]{Vi} for the unstable manifolds for {\em special} parabolic skew-product of $\mathbb{C}^2$. Vivas used this parametrization as the main tool to prove the analogue of Theorem \ref{thm:5} for special parabolic skew-product. However, this construction does not allow to construct a wandering Fatou component in a neighbourhood of the parabolic invariant fiber.

In \cite{ABDPR}, together with Astorg, Dujardin and Peters, we proved the existence of polynomial skew-products of $\mathbb{C}^2$, extending to holomorphic endomorphisms of $\mathbb{P}^2(\mathbb{C})$, having a wandering Fatou component. The key tool consists in using parabolic implosion techniques on polynomial skew-products, and this idea was initially suggested by Lyubich. The main strategy is to combine slow convergence to an invariant parabolic fiber and parabolic transition in the fiber direction, to produce orbits shadowing those of the so-called Lavaurs map, that we defined in the previous section.

\begin{theorem}[\cite{ABDPR}]
There exists a holomorphic endomorphism $F: \mathbb{P}^2(\mathbb{C})\to \mathbb{P}^2(\mathbb{C})$, induced by  a polynomial skew-product mapping
$F: \mathbb{C}^2\to \mathbb{C}^2$, having a wandering Fatou component. More precisely, let $f: \mathbb{C}\to \mathbb{C}$ and $g: \mathbb{C}\to \mathbb{C}$ be polynomials of the form
\begin{equation}
\label{eq:fg}
f(z) = z+z^2+{\rm O}(z^3)~~\hbox{and}~~g(w) = w-w^2+{\rm O}(w^3).
\end{equation}
If the Lavaurs map $\mathcal{L}_f: \mathcal{B}_f\to \mathbb{C}$ has an attracting fixed point, then the skew-product 
$F: \mathbb{C}^2\to \mathbb{C}^2$ defined by
\begin{equation}
\label{eq:skew}
F(z,w) := \left(f(z)+\frac{\pi^2}{4} w,g(w)\right)
\end{equation}
has a wandering Fatou component.
\label{thm:wandering}
\end{theorem}

The orbits in these wandering  Fatou components are bounded and the approach used in the proof is essentially local.  
Notice that if $f$ and $g$ have the same degree, $F$ extends to a holomorphic endomorphism of $\mathbb{P}^2(\mathbb{C})$.
Moreover we can obtain examples in arbitrary dimension $k\geq 2$ by simply considering   products mappings of the form  $(F,Q)$, where $Q$ has a fixed Fatou component.

A first step in the proof of Theorem \ref{thm:wandering} is to find a parabolic polynomial $f$ whose Lavaurs map $\mathcal{L}_f$ admits an attracting fixed point. 

\begin{proposition}[{\cite[Proposition B]{ABDPR}}]
Let $f: \mathbb{C}\to \mathbb{C}$ be the cubic polynomial defined by
\[
f(z)
=
z+z^2+az^3
\quad\text{with}\quad a\in \mathbb{C}.
\]
If $r>0$ is sufficiently close to $0$ and $a$ belongs to the disk $D(1-r,r)$,  then  the Lavaurs map ${\cal L}_f: {\cal B}_f\to \mathbb{C}$ admits an attracting fixed point.
\label{prop:main}
\end{proposition}

\begin{proof}[Idea of the proof] 
We consider
\[
{\cal U}_f:=\psi_f^{-1}({\cal B}_f)
{\quad\text{and}\quad} 
{\cal E}_f:=\varphi_f\circ \psi_f: {\cal U}_f\to \mathbb{C}.\]
The open set ${\cal U}_f$ contains an upper half-plane and a lower half-plane, and it is invariant under $T_1$.  Like the Lavaurs map, the map  ${\cal E}_f$ commutes with $T_1$, therefore ${\cal E}_f-\mathop{\rm Id}\nolimits$ is periodic of period 1 and admits a Fourier expansion in a upper half-plane: 
\[{\cal E}_f(Z) = Z + \sum_{k\geq 0} c_k e^{2\pi i k Z}.\]
Using the expansion of $\varphi_f$ and $\psi_f$ near infinity, we obtain with an elementary computation: 
\begin{align*}
{\cal E}_f(Z) = \varphi_f\circ \psi_f(Z) &= Z + (1-a)\log(-Z)+{ o}(1)\\
&\quad\quad\,-(1-a)\log\bigl(Z + (1-a)\log(-Z)+{\rm o}(1)\bigr) + { o}(1)\\
&= Z +(1-a)\log(Z)-\pi i (1-a)-(1-a)\log(Z) + { o}(1)\\
&= Z-\pi i(1-a)+{ o}(1),
\end{align*}
and so $c_0 = -\pi i(1-a)$.

Thanks to a more elaborate argument, based on the notion of finite type analytic map introduced by Adam Epstein, it is possible to prove that:
\[{\cal E}_f(Z) = Z-\pi i(1-a) + c_1e^{2\pi i Z} + { o}(e^{2\pi i Z})\quad \text{with}\quad c_1\neq 0.\]
It then follows that for $a\neq 1$ close to $1$, ${\cal E}_f$ has a fixed point $Z_f$ with multiplier $\rho_f$ satisfying
\[c_1e^{2\pi i Z_f}\sim \pi i(1-a)\quad \text{and}\quad \rho_f-1\sim 2\pi i c_1 e^{2\pi i Z_f} \sim -2\pi^2(1-a)\quad \text{as}\quad a\to 1. \]
Therefore, for $r>0$ sufficiently close to $0$ and $a\in D(1-r,1)$, the multiplier $\rho_f$ belongs to the unit disk and $Z_f$ is an attracting fixed point of ${\cal E}_f$. This concludes the proof since $\psi_f: {\cal U}_f\to {\cal B}_f$ semi-conjugates ${\cal E}_f$ to ${\cal L}_f$, and so, the fact that $Z_f$ is an attracting fixed point of ${\cal E}_f$ implies that the point $\psi_f(Z_f)$ is an attracting fixed point of ${\cal L}_f$.
\end{proof}

The key result in the proof of Theorem \ref{thm:wandering} relies on a non-autonomous analogue of Lavaurs estimates in the setting of skew-products.

Let ${\cal B}_f$ and ${\cal B}_g$ be the parabolic basins of $0$ under iteration of respectively $f$ and $g$.  One of the key points is to choose $(\widetilde z_0,\widetilde w_0)\in \mathcal{B}_f\times \mathcal{B}_g$ so that the first coordinate of~$F^{\circ m}(\widetilde z_0, \widetilde w_0)$ returns infinitely many times close to the attracting fixed point of $\mathcal{L}_f$. The proof is designed so that the return times are the integers $n^2$ for $n\geq n_0$. Therefore, we need to analyze the orbit segment between $n^2$ and  $(n+1)^2$, which is of length $2n+1$.

\begin{proposition}[\cite{ABDPR}]
As $n\to +\infty$, the sequence of maps
\[
\mathbb{C}^2\ni (z,w) \mapsto F^{  2n+1}\bigl(z,g^{  n^2}(w)\bigr)\in \mathbb{C}^2
\]
converges locally uniformly in $\mathcal{B}_f\times \mathcal{B}_g$ to the map
\[\mathcal{B}_f\times \mathcal{B}_g\ni (z,w)\mapsto \bigl( \mathcal{L}_f(z),0\bigr)\in \mathbb{C}\times\{0\}.\]\label{pr:key}
\end{proposition}

\begin{proof}[Idea of the proof]
Let $\mathcal{B}_g$ the parabolic basin of $0$ under iteration of $g$. For all $w\in \mathcal{B}_g$, the orbit $g^{\circ m}(w)$ converges to~0 like $1/m$. We want to analyze the behaviour of $F$ starting at $\bigl(z,g^{\circ n^2}(w)\bigr)$ during $2n+1$ iterates. For large $n$, the first  coordinate of $F$ along this orbit segment
is approximately
$$
f(z)+ \varepsilon^2~~~\hbox{with}~~~\frac{\pi}{\varepsilon} \simeq 2n.
$$
As we saw in the previous section, a rough statement of Lavaurs Theorem from parabolic implosion gives us that if $\frac{\pi}{\varepsilon} = 2n$, then  for large  $n$,
the $(2n)^{\rm th}$ iterate of $f(z)+ \varepsilon^2$ is approximately equal to $\mathcal{L}_f(z)$ on $\mathcal{B}_f$.

Our setting  is different since in our case $\varepsilon$ keeps decreasing along the orbit. Indeed on the first coordinate  we are taking the composition of $2n+1$ transformations of the form
$$
f(z)+ \varepsilon_k^2~~\hbox{with}~~\frac{\pi}{\varepsilon_k} \simeq 2n+ \frac{k}{n}~~{\hbox{and}}~~1 \leq k \leq    2n+1 .
$$
The key step in the proof of the statement consists in a detailed analysis of this non-autonomous situation, proving that the decay of $\varepsilon_k$ is counterbalanced
by taking {\em exactly} one additional iterate of $F$. 
\end{proof}

With this proposition in hand, the proof of Theorem \ref{thm:wandering} is easily completed. 

\begin{proof}[Proof of Theorem~\ref{thm:wandering}] Let $\xi$ be an attracting fixed point of $\mathcal{L}_f$ and let $V\subset {\cal B}_f$ be a disk centered at $\xi$ and such that ${\cal L}_f(V)$ is compactly contained in $V$. Therefore ${\cal L}_f^{\circ k}(V)$ converges to $\xi$ as $k\to+\infty$. Let $W \Subset {\cal B}_g$ be an arbitrary disk.

Thanks to Proposition \ref{pr:key}, there exists  $n_0 \in \mathbb{N}$ such that for every $n\geq n_0$,
\[ 
\pi_1 \circ  F^{\circ (2n+1)}(V \times g^{\circ n^2}(W))\Subset V,
\]
where $\pi_1: \mathbb{C}^2\to \mathbb{C}$ denotes the projection on the first coordinate, 
$\pi_1(z,w):=z$.

Let $U$ be a connected component of the open set $F^{-n_0^2}\bigl(V\times g^{\circ n_0^2}(W)\bigr)$. Then for every integer $n\geq n_0$, we have
\begin{equation}
\label{eqincl}
 F^{\circ n^2} (U)\subseteq V\times
g^{\circ n^2}(W).
\end{equation}
In fact, this holds by assumption for $n=n_0$.  Now if the inclusion is true for some $n\geq n_0$, then
\begin{align*}
\pi_1  \circ F^{\circ (n+1)^2} (U)& = \pi_1 \circ F^{\circ (2n+1)}
\left( F^{\circ n^2} (U) \right)
\\& \subset \pi_1 \circ F^{\circ (2n+1)} \left(V\times g^{\circ n^2}(W)\right)\subset V,
\end{align*}
from which \eqref{eqincl} follows. This yields that the sequence $(F^{\circ n^2})_{n\geq 0}$ is uniformly bounded, and hence normal, on $U$. Moreover, any cluster value of this sequence of maps is constant and of the form $(z,0)$ for some $z\in V$, and $(z,0)$ is a limit value (associated to a subsequence $(n_k)_{k\in\mathbb{N}}$) if and only if $\bigl({\cal L}_f(z),0\bigr)$ is a limit value (associated to the subsequence $(1+n_k)_{k\in\mathbb{N}}$). Therefore the set of cluster limits is totally invariant under
 ${\cal L}_f:  V\to V$, and so it must coincide with the attracting fixed point $\xi$ of ${\cal L}_f$.
Therefore the sequence $(F^{\circ n^2})_{n\geq 0}$ converges locally uniformly to $(\xi,0)$ on~$U$.

The sequence $(F^{\circ m})_{m\in\mathbb{N}}$ is locally bounded  on $U$ if and only if there exists a subsequence  $(m_k)_{k\in\mathbb{N}}$ such that $(F^{\circ m_k}|_U)_{k\in\mathbb{N}}$  has the same property. In fact, since $ \overline {W}$ is compact, there exists $R>0$ such that if $|z|> R$, then for every  $w\in W$, $(z, w)$ escapes locally uniformly to infinity under iteration. The domain $U$ is therefore contained in the Fatou set of $F$.

Let $\Omega$ be the component of the Fatou set ${\cal F}_F$ containing $U$. 
For any integer $j\ge 0$, the sequence of maps $(F^{\circ n^2+j})_{n\in\mathbb{N}}$ converges locally uniformly to $F^{\circ j}(\xi,0) = (f^{\circ j}(\xi),0)$ on $U$ and hence on $\Omega$. Therefore the sequence $(F^{\circ n^2})_{n\in\mathbb{N}}$ converges locally uniformly to $(f^{\circ j}(\xi),0)$ on $F^{\circ j}(\Omega)$. If $i,j$ are nonnegative integers such that $F^{\circ i}(\Omega)= F^{\circ j}(\Omega)$, then $f^{\circ i}(\xi)= f^{\circ j}(\xi)$, and so $i=j$ because $\xi$ cannot be pre-periodic under iteration of $f$, since it belongs to the parabolic basin $\mathcal{B}_f$. This proves that
$\Omega$ is not (pre-)periodic under iteration of $F$, and so it is a wandering Fatou component for $F$.
\end{proof}

We end this subsection with some explicit examples satisfying the assumption of Theorem~\ref{thm:wandering}.

\begin{example}{Example}
As a consequence of Proposition~\ref{prop:main} we obtain that if $f: \mathbb{C}\to \mathbb{C}$ is the cubic  polynomial $f(z) = z+z^2+az^3$, and $g$ is as in \eqref{eq:fg}, then the polynomial skew-product $F$ defined in \eqref{eq:skew} admits a wandering Fatou component  for  $r>0$ sufficiently small and $a\in D(1-r, r)$.
\label{ex:complex}
\end{example}

It is also interesting to search for real polynomial mappings with wandering Fatou domains intersecting $\mathbb{R}^2$. We also have such examples.

\begin{example}{Example \cite[Proposition C]{ABDPR}}
\label{ex:real}
Let $f: \mathbb{C}\to \mathbb{C}$ be the degree $4$ polynomial  defined by
$$
f(z):=z+z^2+bz^4~\hbox{with}~b\in \mathbb{R}.
$$
There exist parameters $b\in (-8/27,0)$ such that for $g$ as in \eqref{eq:fg},  the
polynomial skew-product $F$ defined in \eqref{eq:skew} has a wandering Fatou component intersecting~$\mathbb{R}^2$.
\end{example}

Astorg, Boc-Thaler and Peters constructed in \cite{ABP} a second example of polynomial skew-products, arising from similar techniques, but with distinctly different dynamical behaviour. Instead of wandering domains arising from a Lavaurs map with an attracting fixed point, they constructed a domain arising from a Lavaurs map with a fixed point of Siegel type. 

In \cite{HaPe}, Hahn and Peters constructed polynomial automorphisms of $\mathbb{C}^4$ with wandering Fatou components. Their four-dimensional automorphisms lie in a one-parameter family, depending on a parameter $\delta\in\mathbb{C}\setminus\{0\}$, and degenerating to the two-dimensional polynomial map in equation \eqref{eq:fg} of Theorem \ref{thm:wandering} as $\delta$ converges to~$0$.

Berger and Biebler recently proved in \cite{BergerBiebler} the existence of a locally dense set of polynomial automorphisms of $\mathbb{C}^2$ with real coefficients, having a wandering Fatou component. These Fatou components have nonempty real trace and their statistical behaviour is historic with high emergence. The proof is based on a geometric model for parameter families of surface real mappings.  	

\subsection{Elliptic invariant fiber}
\label{subsec:8.5}

With Peters we investigated in \cite{PR} the case of invariant fibers at the center of a Siegel disk. More precisely, we considered a polynomial skew-product of the form \eqref{eq:1} having an {\em elliptic invariant fiber}. As before, we can assume without loss of generality that the invariant fiber is $\{w=0\}$ and so we have $g(0)=0$ and $g'(0) = e^{2\pi i \theta}$ with $\theta \in\mathbb{R}\setminus\mathbb{Q}$. We assume that the origin belongs to a Siegel disk for $g$ and hence $g$ is locally holomorphically linearizable near $w = 0$. Therefore, up to a local change of coordinates we may assume that $F$ is of the form
$$
F(z,w) = (f(z,w), \lambda \cdot w),
$$
where $f_w(z) := f(z,w)$ is a polynomial in $z$ with coefficients depending holomorphically on $w$. We assume that the degree of the polynomial $f_w$ is constant near $\{w = 0\}$, and at least $2$.

In this case we only have a partial answer to Question 1. In fact, while we already know that the attracting Fatou components of $f_0$ always bulge, the general situation appears to be more complicated as there might be resonance phenomena. In \cite{PR} we proved the following local result, implying that all parabolic Fatou components of a polynomial skew-product with an elliptic invariant fiber bulge if the rotation number satisfies the Brjuno condition.

\begin{proposition}[{\cite[Proposition 2]{PR}}]
Let $F$ be a holomorphic skew-product of the form
\begin{equation}
F(z,w) = (f_w(z),g(w))
\end{equation}
with $g(w) = \lambda  w + O(w^2)$, $f_0(0)=0$, and $f_0'(0)=1$. Assume $\lambda$ is a Brjuno number. If $f_0(w)\equiv w$, then $F$ is holomorphically linearizable at the origin. If $f_0(z) = z + f_{0,k+1} z^{k+1} + O(z^{k+2})$ with $f_{0,k+1}\ne0$ for some $k \ge 1$, then for any $h\ge 0$ there exists a local holomorphic change of coordinates near the origin conjugating $F$ to a map of the form $\widetilde{F}(z,w) = (\widetilde{f}(z,w),\lambda w)$ satisfying
\begin{equation}
\widetilde{f}(z,w) = z + f_{0,k+1} z^{k+1} + \cdots + f_{0,k+h+1} z^{k+h+1} + \sum_{j\ge h}z^{k+j+2}\alpha_{k+j+2}(w),
\end{equation}
where, for $j\ge h$, $\alpha_{k+j+2}(w)$ is a holomorphic function in $w$ such that $\alpha_{k+j+2}(0) = f_{0,k+j+2}$.
\label{localcoordinates}
\end{proposition}

We recall that a complex number is called \emph{Brjuno} (see \cite{Brjuno1} for more details) if
\begin{equation}\label{eq:brjuno}
\sum_{k=0}^{+\infty}{\frac{1}{2^k}}\log{\frac{1}{\omega(2^{k+1})}}<+\infty\;,
\end{equation}
where $\omega(m) = \min_{2\le k\le m} |\lambda^k - \lambda|$ for any $m\ge 2$, and a quadratic polynomial $\lambda w + w^2$ is linearizable near $0$ if and only if $\lambda$ is Brjuno (see \cite{Brjuno1, Brjuno2, Yoccoz}). For higher degree polynomials the Brjuno condition \eqref{eq:brjuno} is sufficient, but necessity is still an open question. 

The proof of the previous proposition makes use of a procedure inspired by the Poincar\'e-Dulac normalization process \cite[Chapter 4]{arnold} aiming to conjugate the given polynomial skew-product to a skew-product in a simpler form. At each step of the usual Poincar\'e-Dulac normalization procedure we can use a polynomial change of coordinates to eliminate all non-resonant monomials of a given degree. A similar idea is used here, and thanks to the skew-product structure of the germ and the Brjuno assumption we are able to eliminate all non-resonant terms of a given degree in the powers of $z$ by means of changes of coordinates that are polynomial in $z$ with holomorphic coefficients in $w$. It turns out that there are some differences in between the first degrees as it is pointed out in \cite[Section 2]{PR}.

The arithmetic Brjuno condition is strongly needed in the proof of Proposition \ref{localcoordinates} and it is natural to ask whether parabolic Fatou components of $f_0$ always bulge, or it is possible to characterize when they bulge. We do not have a complete answer to such question. However we can prove the following result.

\begin{proposition}[{\cite[Proposition 10]{PR}}]
Let $F(z,w) = (f_z(w), g(w))$ be a holomorphic skew-product with an elliptic linearizable invariant fiber. If $F$ does not admit a holomorphic invariant curve on the invariant fiber, then the parabolic Fatou components of $f_0$ do not bulge.
\end{proposition}

Moreover, we can construct examples showing that the Brjuno condition cannot be completely omitted. In fact by taking {\em Cremer} numbers, that is $\lambda\in\mathbb{C}$ with $|\lambda|=1$ and such that
$$
\limsup_{m\to\infty} \frac{1}{m} \log\frac{1}{\omega(m)} = +\infty,
$$
we can construct examples of polynomial skew-product not admitting a holomorphic invariant curve on the invariant fiber. An elementary example is the following polynomial skew-product
$$
F(z,w) = (z+w+zw, \lambda w)
$$
with $\lambda$ a Cremer number (see {\cite[Example 8]{PR}} for details). More generally, arguing as in Cremer's example \cite{Cremer2} we can show (see \cite[Proposition 9]{PR}) the existence of holomorphic skew-products without invariant holomorphic curves of the form $\{w=\varphi(z)\}$, having an elliptic invariant fiber that is not point-wise fixed.

Describing the general situation can also be complicated by resonance phenomena. For example, an invariant fiber at the center of a Siegel disk was used in \cite{BFP} to construct a non-recurrent Fatou component with limit set isomorphic to a punctured disk, and in that construction the invariant fiber also contains a Siegel disk, but with opposite rotation number. Moreover, it might happen that Fatou components on the invariant fiber do not bulge. For example, we can consider the skew-product
$$
F(z,w) = (\lambda z(1+ a zw), \lambda^{-1} w),
$$
with $a\in\mathbb{C}^*$, $\lambda= e^{2\pi i \theta}$ and $\theta\in\mathbb{R}\setminus\mathbb{Q}$. We have $F(z,0) = (\lambda z, 0)$, but the Siegel disk around the origin in $\{w=0\}$ is not bulging, and in fact it follows from \cite{BZ} and \cite{BRZ} that there exists a Fatou component of parabolic type having on its boundary the origin of $\mathbb{C}^2$, which is fixed by $F$.

We also have an answer to Question 2, under the assumption that the multiplier at the elliptic invariant fiber is Brjuno and all critical points of the polynomial acting on the invariant fiber lie in basins of attracting or parabolic cycles.

\begin{theorem}[{\cite[Theorem 1]{PR}}]
Let $F$ be a polynomial skew-product of the form
\begin{equation}
F(z,w) = (f_w(z),g(w)),
\end{equation}
and let $\{w=w_0\}$ be an elliptic invariant fiber with multiplier $\lambda$.
If $\lambda$ is Brjuno and all critical points of the polynomial $f_{w_0}$ lie in basins of attracting or parabolic cycles, then all Fatou components of $f_{w_0}$ bulge, and there is a neighbourhood of the invariant fiber $\{w={w_0}\}$ in which the only Fatou components of $F$ are the bulging Fatou components of $f_{w_0}$. In particular there are no wandering Fatou components in this neighbourhood.
\label{thm:mainPR}
\end{theorem}

\begin{proof}[Idea of the proof]
The proof relies on the fact that, in dimension 1, if every critical point of the polynomial either (1) lies in the basin of an attracting periodic cycle, (2) lies in the basin of a parabolic periodic cycle, or (3) lies in the Julia set and after finitely many iterates is mapped to a periodic point, then it is possible to construct a conformal metric $\mu$, defined in a backward invariant neighbourhood of the Julia set minus the parabolic periodic orbits, so that $f_c$ is expansive with respect to this metric. It then follows that there can be no wandering Fatou components.

The key step in the proof is that for fibers $\{w = \widetilde w_0\}$ sufficiently close to the invariant fiber $\{w= w_0\}$ we can define conformal metrics $\mu_{\widetilde w_0}$ that depend continuously on $\widetilde w_0$, and so that $F$ acts expansively with respect to this family of metrics (see \cite[Section 3]{PR} for details). That is, for a point $(\widetilde z_0, \widetilde w_0) \in \mathbb C^2$ lying in the region where the metrics are defined, and for a non-zero tangent vector $\xi \in T_{\widetilde z_0}(\mathbb C_{\widetilde w_0})$ we have that
$$
\mu_{\widetilde w_1} (\widetilde z_1, df_{\widetilde w_0} \xi) > \mu_{\widetilde w_0}(\widetilde z_0, \xi),
$$
where $(\widetilde z_1,\widetilde w_1) = F(\widetilde z_0, \widetilde w_0)$. Therefore, if $(\widetilde z_0,\widetilde w_0)$ does not lie in one of the bulging Fatou components and its orbit $(\widetilde z_n,\widetilde w_n)$ remains in the neighbourhood where the expanding metrics are defined, for any tangent vector $\xi \in T_{\widetilde z_0}(\mathbb C_{\widetilde w_0})$ we have that
$$
\mu_{\widetilde w_n} df_{\widetilde z_{n-1}} \cdots df_{\widetilde z_0} \xi \rightarrow \infty,
$$
from which it follows that the family $(F^{\circ n})_{n\in\mathbb{N}}$ cannot be normal on any neighbourhood of $(\widetilde z_0,\widetilde w_0)$. This proves that there are no Fatou components but the bulging components near the invariant fiber $\{w=w_0\}$.
\end{proof}

It is unclear whether the same techniques could be used to deal with pre-periodic critical points lying on the Julia set. The difficulty is that the property that critical points are eventually mapped onto periodic cycles is not preserved in nearby fibers. 

\section*{Acknowledgements}
We thank John H. Hubbard for letting us use the book in preparation \cite{BuffHubbard} to write this survey. The first seven sections are adapted from the first chapter of this book. 
We also thank Marco Abate and Fabrizio Bianchi for useful comments on the first draft of these notes.

\vfill\eject

\end{document}